\documentclass[12pt]{article}

\usepackage[hyphens]{url}
\usepackage{hyperref}
\usepackage[hyphenbreaks]{breakurl}
\hypersetup{
           breaklinks=true,
           colorlinks=true,
    		linkcolor=black,
    		citecolor=black,
    		filecolor=black,
    		urlcolor=black,
        }

\usepackage[utf8]{inputenc}
\usepackage{mathtools}
\usepackage{amsmath}
\usepackage{amsthm}
\usepackage{amssymb}
\usepackage{lpform}
\usepackage{tikz}
\usepackage{graphicx, nicefrac}
\usetikzlibrary{arrows,%
                petri,%
                topaths}%
\usepackage{tkz-berge}
\usepackage{caption}
\usepackage{xurl}
\usepackage{authblk}
\usepackage{multirow}
\usepackage{lmodern}
\usepackage{enumitem}
\usetikzlibrary{positioning}

\newtheorem{prop}{Proposition}[section]
\newtheorem{cor}{Corollary}[section]
\newtheorem{lem}{Lemma}[section]

\usepackage{apptools}
\usepackage[authoryear]{natbib}
\usepackage{fullpage}
\AtAppendix{\counterwithin{prop}{section}}

\makeatletter
\renewcommand*{\NAT@spacechar}{~}

\newbox\xdottedarrow@box
\setbox\xdottedarrow@box\hbox
  {%
    \begin{tikzpicture}
      \draw[dotted,->, very thick, cyan] (0,0) -- (2em,0);
    \end{tikzpicture}
  }
\newcommand*\xdottedarrow
  {\relax\ifmmode\expandafter\xdottedarrow@m\else\expandafter\xdottedarrow@t\fi}
\newcommand*\xdottedarrow@t[1][1.5em]
  {\resizebox{#1}{!}{\raisebox{.5ex}{\usebox\xdottedarrow@box}}}
\newcommand*\xdottedarrow@m[1][]
  {%
    \if\relax\detokenize{#1}\relax
      \mathchoice
        {\xdottedarrow@t}
        {\xdottedarrow@t}
        {\xdottedarrow@t[1.1em]}
        {\xdottedarrow@t[0.9em]}%
    \else
      \xdottedarrow@t[#1]%
    \fi
  }

\newbox\xdotdasharrow@box
\setbox\xdotdasharrow@box\hbox
  {%
    \begin{tikzpicture}
      \draw[very thick, dash dot,->, teal] (0,0) -- (2em,0);
    \end{tikzpicture}
  }
\newcommand*\xdotdasharrow
  {\relax\ifmmode\expandafter\xdotdasharrow@m\else\expandafter\xdotdasharrow@t\fi}
\newcommand*\xdotdasharrow@t[1][1.5em]
  {\resizebox{#1}{!}{\raisebox{.5ex}{\usebox\xdotdasharrow@box}}}
\newcommand*\xdotdasharrow@m[1][]
  {%
    \if\relax\detokenize{#1}\relax
      \mathchoice
        {\xdotdasharrow@t}
        {\xdotdasharrow@t}
        {\xdotdasharrow@t[1.1em]}
        {\xdotdasharrow@t[0.9em]}%
    \else
      \xdotdasharrow@t[#1]%
    \fi
  }
  
\newbox\xdasharrow@box
\setbox\xdasharrow@box\hbox
  {%
    \begin{tikzpicture}
      \draw[very thick, dashed,->, blue] (0,0) -- (2em,0);
    \end{tikzpicture}%
  }
\newcommand*\xdasharrow
  {\relax\ifmmode\expandafter\xdasharrow@m\else\expandafter\xdasharrow@t\fi}
\newcommand*\xdasharrow@t[1][1.5em]
  {\resizebox{#1}{!}{\raisebox{.5ex}{\usebox\xdasharrow@box}}}
\newcommand*\xdasharrow@m[1][]
  {%
    \if\relax\detokenize{#1}\relax
      \mathchoice
        {\xdasharrow@t}
        {\xdasharrow@t}
        {\xdasharrow@t[1.1em]}
        {\xdasharrow@t[0.9em]}%
    \else
      \xdasharrow@t[#1]%
    \fi
  }
  
\newbox\xregarrow@box
\setbox\xregarrow@box\hbox
  {%
    \begin{tikzpicture}
      \draw[very thick, ->, red] (0,0) -- (2em,0);
    \end{tikzpicture}
  }
\newcommand*\xregarrow
  {\relax\ifmmode\expandafter\xregarrow@m\else\expandafter\xregarrow@t\fi}
\newcommand*\xregarrow@t[1][1.5em]
  {\resizebox{#1}{!}{\raisebox{.5ex}{\usebox\xregarrow@box}}}
\newcommand*\xregarrow@m[1][]
  {%
    \if\relax\detokenize{#1}\relax
      \mathchoice
        {\xregarrow@t}
        {\xregarrow@t}
        {\xregarrow@t[1.1em]}
        {\xregarrow@t[0.9em]}%
    \else
      \xregarrow@t[#1]%
    \fi
  }
\makeatother

\title{The TSP with drones: The benefits of retraversing the arcs}
\author[]{Nicola Morandi\textsuperscript{a,$\star$}, Roel Leus\textsuperscript{a}, Jannik Matuschke\textsuperscript{b}, Hande Yaman\textsuperscript{a}}
\affil[]{\textsuperscript{a}\small Research Centre for Operations Research and Statistics, KU Leuven}
\affil[]{\textsuperscript{b}Research Centre for Operations Management, KU Leuven}
\affil[]{\textsuperscript{$\star$}nicola.morandi@kuleuven.be}

\date{}
\begin{document}

\renewcommand{\lpindent}{\hspace{0pt}}

\maketitle

\begin{abstract}
\footnotesize In the Traveling Salesman Problem with Drones (TSP-mD), a truck and multiple drones cooperate to serve customers in the minimum amount of time. The drones are launched and retrieved by the truck at customer locations, and each of their flights must not consume more energy than allowed by their batteries. Most problem settings in the literature restrict the feasible truck routes to cycles, i.e., closed paths, which never visit a node more than once. \emph{Revisiting} a node, however, may lower the time required to serve all the customers. Additionally, we observe that optimal solutions for the \mbox{TSP-mD} may \emph{retraverse} arcs, i.e., optimal truck routes may contain the same arcs multiple times. We refer to such solutions as \emph{arc-retraversing}, and include them in our solution space by modeling the truck route as a closed walk. We describe Euclidean instances where all the optimal solutions are arc-retraversing. The necessity of arc retraversals does not seem to have been investigated in previous studies, and those that allow node revisits seem to assume that there always exists an optimal solution without arc retraversals. We prove that under certain conditions, which are commonly met in the literature, this assumption is correct. When these conditions are not met, however, excluding arc-retraversing solutions might result in an increase of the optimal value; we identify cases where \emph{a priori} and \emph{a posteriori} upper bounds hold on such increase. Finally, we prove that there is no polynomial-time heuristic that can approximate the metric TSP-mD within a constant factor, unless P=NP. We identify a (non-constant) approximation factor explicitly when the truck can visit all the nodes.
\end{abstract}

\section{Introduction}
Since the pioneering introduction of the Flying Sidekick Traveling Salesman Problem (FSTSP) by \cite{murraychu15}, the scientific literature about applications of drones to routing and parcel delivery has grown at a remarkable pace. At the time of writing, searching the words ``truck drone routing'' by Google Scholar produced $19\, 200$ results, $20\%$ of which from just 2021. An impressive number of surveys on the topic already appeared, e.g., those by \cite{macrina20} and \cite{otto18review}. The interest in drones' applications also comes from public institutions and the private sector. The European Commission forecasts more than $100\, 000$ people employed and an economic impact of over $10$ billion Euros per year in the European drone sector by 2035 \citep{EUcomm}. At the same time, the e-commerce multinational Amazon obtained the approval from the relevant US authority for its Prime Air service ``beyond visual line of sight'' \citep{Amazon}. From the algorithmic point of view, the FSTSP and its generalizations can model a wide range of routing applications with cooperating vehicles, which are not necessarily drones; in principle, any vehicle with limited fuel, or traveling person with limited payload capacity, could play the role of the drone in these problems.

In the FSTSP\footnote{There are numerous similar problem variants studied in literature under the name FSTSP or TSP with a Drone (TSP-D). The generic setting we outlined here under the name FSTSP is equivalent to the one studied by \citet{agatz18} and \citet{roberti21}.}, a truck and a drone cooperate to visit all the customers in a given network in the minimum amount of time.
The drone can only serve one customer per sortie (i.e., drone flight). 
A natural generalization of FSTSP is the TSP with Drones (\mbox{TSP-mD}), where multiple drones are allowed to serve multiple customers per sortie, with the length of each sortie bounded by a limited battery capacity. 
Most studies of FSTSP and \mbox{TSP-mD} impose the additional constraint that the truck cannot visit customers multiple times, with the notable exception of the work by \citet{agatz18}, \cite{bouman18}, and \cite{tang19}.
As pointed out by \citet{roberti21}, allowing such \emph{revisits} poses additional computational challenges which cannot be easily accommodated by many of the existing approaches in literature. Indeed, by adding sufficiently many copies of the nodes to the underlying graph, one can always reduce revisiting truck routes to cycles; this approach, however, does not appear to be computationally viable in practice.
\citet{roberti21} also state that an analysis of the cost savings by revisits had, at the time of their writing, not been conducted yet.

In this paper, we observe that optimal solutions for \mbox{TSP-mD} may not only need to revisit nodes, but also \emph{retraverse} arcs of the underlying directed graph, i.e., in the course of its tour the truck might need to repeatedly travel directly from customer $i$ to customer $j$ for some fixed pair of customers $i, j$. As we show, excluding such \emph{arc-retraversing} solutions can lead to a significant increase in the optimal value.
The necessity of arc retraversals does not seem to have been investigated in previous studies, and those studies that allow node revisits seem to operate under the assumption that there always exists an optimal solution without retraversals.
In fact, the integer programming (IP) formulation proposed in the seminal paper by \citet{agatz18} for FSTSP with node revisits implicitly excludes certain arc-retraversing solutions.
We prove that this implicit assumption is correct under specific conditions, which the FSTSP setting studied by \citet{agatz18} meets. 
However, when these conditions are not satisfied (e.g., when allowing multiple customers to be visited by a single sortie), the optimal value might increase significantly when excluding arc-retraversing solutions. 
We provide asymptotically tight \emph{a priori} (i.e., solution-independent) and \emph{a posteriori} (i.e., solution-dependent) upper bounds on such increase.
In particular, the optimal value can increase by a factor of at most $1+2m$ (where $m$ is the number of drones) in the worst case when excluding arc-retraversing solutions. The same worst-case increase holds when excluding node-revisits, giving a partial answer to the question raised by \cite{roberti21}.
Finally, we provide an approximation algorithm whose approximation guarantee depends on the speed and the number of available drones, and show that unless P=NP, no approximation algorithm can obtain a guarantee that does not depend on these two parameters.

The remainder of the paper is organized as follows. A review of relevant 
work is presented in Section~\ref{sec_literature}, while the TSP-mD itself is formally defined in Section~\ref{sec_prob_descr}. In Section~\ref{sec_arcretrav}, we describe problem settings where it suffices to consider solutions that are not arc-retraversing. We establish the aforementioned upper bounds on the increase of the optimal value in Section~\ref{sec_notincluding}. Finally, Section~\ref{sec_approx} contains the proofs of our approximability results. We conclude with Section~\ref{sec_conclusion}.
To prove the results of Section~\ref{sec_arcretrav}, we solve a number of instances via an MILP formulation; we describe the relevant instances and provide the MILP formulation in Appendices~\ref{app_inst} and \ref{app_model}, respectively.

\section{Related literature}
\label{sec_literature}
In this section, we review the drone routing literature that specifically addressed truck-drone(s) operation problems with exact methods. We further focus on settings with a single truck that visits or delivers to customers in parallel to the drone, and where the completion time is minimized.

For a fast entry point to the literature, not exclusively from the operations research perspective, we refer the reader to the surveys (in alphabetical order) by \cite{boysen21survey}, \cite{ding21survey}, \cite{li21survey}, \cite{macrina20survey}, \cite{merkert20survey}, \cite{otto18review}, \cite{persson21survey}, and \cite{roca19survey}. For surveys with a focus on drone applications to routing and parcel delivery, we refer the reader to \cite{coutinho18review}, \cite{khoufi19review}, \cite{chung20review}, \cite{macrina20}, \cite{rojas21review}, and \cite{moshref21review}, in chronological order of publication from 2018 to 2021. Finally, we mention the instructive overview of the challenges ahead for drone-aided routing provided by \cite{poikonen21}.

The FSTSP was introduced by \cite{murraychu15}. They provided an MILP model, and solved the problem via heuristic methods. A two-stage decomposition for solving the FSTSP was developed by \cite{yurek18}, by which they were able to solve instances with $12$ nodes to optimality within one hour of computations. \cite{dellamico21opletters} provided two novel formulations for the FSTSP, and further refined them in a follow-up paper \citep{dellamico21inttransopres}; they also proposed a branch-and-bound algorithm capable of solving instances with up to $19$ vertices in one hour of CPU time \citep{dellamico21omega}. Recently, \cite{freitas21} proposed a novel MILP formulation for the FSTSP, by which they solved instances with up to 10 nodes in an average of less than two minutes.

\cite{jeong19} proposed a mathematical model to solve a generalization of the FSTSP that incorporates circle-shaped no-fly zones for drones, and parcel weights. They reported exact solutions for instances with up to 10 customers. \cite{luo19}, \cite{gonzalez20}, and \cite{ha21} tackled variants of the FSTSP where sorties can contain multiple customers; the first authors tailored their variant to traffic patrolling applications. \cite{vasquez21} and \cite{boccia21} considered a variant of the FSTSP that further allows \emph{loops}, i.e., sorties with the same launch and landing location. They both solved instances with up to $20$ vertices in a reasonable amount of time, the former by a Benders decomposition and the latter via branch-and-cut.
In a follow-up paper, \cite{boccia21C} also solved the FSTSP by combining a branch-and-cut with a column generation procedure. \cite{jeon21} introduced a variant where both delivery and pick-up demands are met by allowing every sortie to visit one delivery and one pick-up location (in this order) before landing on the truck. They solved instances with up to nine customers via an MILP model, within 30 minutes of CPU time on average.

\cite{agatz18} introduced a variant of the FSTSP where the drone can perform loop sorties and the truck can revisit a customer. They named this variant TSP with a Drone (\mbox{TSP-D}), and proposed a model that contains a large number of binary variables, one for each feasible truck-drone joint operation. The same authors devised a dynamic programming approach in \cite{bouman18}, and solved instances with up to 16 nodes, within three hours of computations on average. \cite{tang19} proposed a constraint programming approach for the \mbox{TSP-D}, and solved instances with up to $18$ nodes, in an average of less than 10 minutes of computations. To the best of our knowledge, these three papers are the only ones in the literature that solved a variant of the FSTSP with a single truck and node revisits. In particular, \cite{agatz18} is the only study that proposed an IP formulation for the problem; in Section~\ref{sec_prob_descr}, we describe arc-retraversing solutions that are not allowed by this formulation.

There is no general consensus in the literature on whether the \mbox{TSP-D} should allow node revisits, as per the original definition in \cite{agatz18}. In the remainder of this section, we classify the references by their own definition of the \mbox{TSP-D}. Some studies, e.g., \cite{schermer20} and \cite{eladle21}, opted for excluding node revisits in the \mbox{TSP-D}. The former proposed MILP formulations capable of directly solving instances with up to $10$ customers within one hour of CPU time, and up to $20$ customers when embedded in a branch-and-cut algorithm.  The latter provided an MILP model that solved instances with up to $24$ nodes. \cite{zhu22} tackled a variant of the \mbox{TSP-D} where the truck is also an electric vehicle, and must visit recharge stations periodically. Their electric truck is only allowed to revisit the locations corresponding to recharging stations. They developed a branch-and-price algorithm by which they solved instances with up to 10 nodes in one hour. Finally, \cite{roberti21} proposed a branch-and-price algorithm to effectively solve \mbox{TSP-D} instances with up to $39$ nodes within one hour of CPU time.

The FSTSP was soon generalized to the multiple drones case. In 2019, \cite{seifried19} proposed an MILP model based on vehicle flows. \cite{murray20} solved instances with up to eight nodes via an MILP model within one hour of computations. In their setting, the launch and retrieval times for the drones are not negligible. \cite{dellamico21networks} tackled a further variant where the drones are allowed to wait for the truck by hovering, and their retrieval gives rise to a nested scheduling problem at any customer location. They provided four formulations and solved instances with 10 customers to optimality in one hour. \cite{jeong19conf} and \cite{luo21} further allowed multiple customers in a single sortie, and solved instances with up to 10 customers in a reasonable amount of time. The latter further took the drones' payload into account in the battery energy consumption. \cite{cavani21} identified a number of symmetry-breaking and valid inequalities, and solved instances with up to $25$ nodes to optimality via branch-and-cut, with a time limit of two hours.

For sake of completeness, we also mention a number of studies on the further generalization to the multiple trucks case: among other ones, \cite{kitjacharoenchai19}, \cite{bakir20}, \cite{tamke21}, and \cite{zhou22}. From the theoretical side, \cite{wang17} provided several worst-case bounds on the optimal value, involving the number of vehicles and their relative speed.

\section{Problem statement}
\label{sec_prob_descr}
The TSP-mD can be defined as follows. Let $N$ be the set of nodes including the customer locations and the depot (denoted by $0$), and $A$ be the set of arcs $\left( i,j\right)$ for any pair of distinct nodes $i,j\in N$. The resulting graph $G=\left( N,A \right)$ is complete and directed. A non-capacitated truck and $m$ identical drones cooperate to serve all the customers in $N$. We denote by $N^{\mathit{dr}}\subseteq N$ and $N^{\mathit{tr}}\subseteq N$ the subsets of the locations that can be served by the drones and by the truck, respectively. The depot $0$ belongs to $N^{\mathit{tr}}$, and $N^{\mathit{dr}}\cup N^{\mathit{tr}}=N$. Notice that if a node $i\in N\smallsetminus N^{\mathit{tr}}$, then the truck cannot traverse any of the arcs that are incident to $i$.

We associate two distinct metrics $\ell$ and $\ell'$ with the arcs in $A$, representing the time it takes for the truck and for the drones, respectively, to traverse the arcs. In particular, by defining $\ell$ and $\ell'$ as metrics, we implicitly require that they are symmetric, i.e., $\ell_{ij}=\ell_{ji}$ and $\ell'_{ij}=\ell'_{ji}$ for all $\left( i,j\right)\in A$. 
We assume without loss of generality that the truck travels at unit speed: therefore, $\ell_{ij}$ also measures the length of arc $\left( i,j\right)$, for every $\left( i,j\right)\in A$. We denote the maximum speedup of the drones compared to the truck by
\begin{equation}
\alpha = \max\!\left\lbrace \nicefrac{\ell_{ij}}{\ell'_{ij}}\; :\; \left( i, j\right)\in A,\;\ell'_{ij}>0\text{, and }i,j\in N^{\mathit{dr}}\cap N^{\mathit{tr}}\right\rbrace.
\end{equation}
When it further holds that $\ell_{ij}=\alpha\cdot\ell'_{ij}$ for all distinct $i,j\in N^{\mathit{dr}}\cap N^{\mathit{tr}}$, we say that the two metrics are \emph{proportional}.

The truck route consists of a closed walk that starts and ends at $0$, and serves all the nodes contained therein. The drones can be independently launched onto airborne routes, hereafter referred to as \emph{sorties}, and retrieved by the truck at the nodes on its route. Accordingly, we represent a sortie~$\pi$ by a tuple of nodes $\left( i_{1}, \dots, i_{r}\right)$, with $r\geq 3$, where $i_{1}, i_{r}\in N^{\mathit{tr}}$ and $i_{2}, \ldots, i_{r-1}\in N^{\mathit{dr}}$. The nodes $i_{1}$ and $i_{r}$ represent the starting and ending nodes of $\pi$, respectively. The set $\left\lbrace i_{2}, \ldots, i_{r-1}\right\rbrace$, which we denote by $N\!\left[ \pi\right]$, represents the nodes that are served by the drone while flying along $\pi$. The drone can serve multiple customers in a single sortie. The amount of energy required to perform a sortie $\pi=\left( i_{1}, \dots, i_{r}\right)$ is given by
\begin{equation}
\label{energy_consumption}
    w^{\mathit{dr}}\cdot\sum_{q=1}^{r-1}\ell'_{i_{q}i_{q+1}}+\sum_{p=2}^{r-1}\left(w_{i_{p}}\cdot\sum_{q=1}^{p-1}\ell'_{i_{q}i_{q+1}}\right),
\end{equation}
where $w^{\mathit{dr}}$ is the weight of a single drone, and $w_{i}\geq 0$ is the payload to serve the customer at node $i$, for every $i\in N\!\left[\pi\right]$. The energy consumed by any sortie must not exceed the maximum value $B>0$ allowed by the battery, and every time a drone lands, its battery is swapped with a fully-charged one in a negligible amount of time. We denote the set of the feasible sorties by $P$. By a slight abuse of notation, we denote by $\ell'_{\pi}$ the duration of a feasible sortie $\pi$, i.e., $\ell'_{\pi}=\sum_{q=1}^{r-1}\ell'_{i_{q}i_{q+1}}$. By setting $w_{i}=0$ for every node $i\in N^{\mathit{dr}}$ in the energy consumption~\eqref{energy_consumption}, we obtain that $\ell'_{\pi}\leq\nicefrac{B}{w^{\mathit{dr}}}$ for every $\pi\in P$; we denote this upper bound on the duration of the feasible sorties by $L$.

Notice that, because of the energy consumption \eqref{energy_consumption}, it might be infeasible to reverse the order of the nodes of a sortie, as illustrated in Figure~\ref{fig:asym_sortie}: with $w^{\mathit{dr}}=10$, 
$B=350$ and a payload $w_{C}=5$ at customer C, the energy consumption of the sortie on the left is $300 + 10\cdot 5=350\leq B$, while the one of its inverted counterpart is $300 + 20\cdot 5=400>B$. We refer to such sorties as \emph{non-invertible}. The invertibility of the sorties is a crucial property in the results of Section~\ref{sec_arcretrav}; in fact, this property is satisfied by most of the settings in the related literature.

\begin{figure}
\begin{minipage}{0.45\textwidth}
\centering
\begin{tikzpicture}[scale=0.3,transform shape]
  \tikzset{VertexStyle/.append style={scale=1.67}}
  \Vertex[x=0.0,y=0.0,L={A}]{A}
  \Vertex[x=10.0,y=0.0,L={B}]{B}
  \Vertex[x=0.0,y=5.77,L={C}]{C}
  
  \tikzstyle{LabelStyle}=[scale=1.67,fill=white,sloped]
  \tikzstyle{EdgeStyle}=[post, color=red]
  \Edge(A)(B)
  
  \tikzstyle{EdgeStyle}=[post, color=blue, dashed]
  \Edge[label=$10$](A)(C)
  \Edge[label=$20$](C)(B)
\end{tikzpicture}
\end{minipage}
\begin{minipage}{0.45\textwidth}
\centering
\begin{tikzpicture}[scale=0.3,transform shape]
  \tikzset{VertexStyle/.append style={scale=1.67}}
  \Vertex[x=0.0,y=0.0,L={A}]{A}
  \Vertex[x=10.0,y=0.0,L={B}]{B}
  \Vertex[x=0.0,y=5.77,L={C}]{C}
  
  \tikzstyle{LabelStyle}=[scale=1.67,fill=white,sloped]
  \tikzstyle{EdgeStyle}=[post, color=red]
  \Edge(B)(A)
  
  \tikzstyle{EdgeStyle}=[post, color=blue, dashed]
  \Edge[label=$20$](B)(C)
  \Edge[label=$10$](C)(A)
\end{tikzpicture}
\end{minipage}
\captionof{figure}{A couple of sorties, one being the inversion of the other. When $w^{\textit{dr}}=10$, $w_{C}=5$ and $B=350$, the sortie on the right is infeasible.}
\label{fig:asym_sortie}
\end{figure}
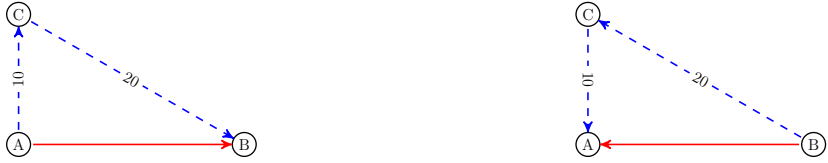

In the TSP-mD, both the truck and the drones are allowed to wait for each other at customer locations. The whole operation is complete when all the customers are served, and both the truck and the drones have returned to the depot. We minimize the completion time. Figure~\ref{fig:feas_sol} shows a feasible solution of an instance with 10 nodes (including the depot~$A$) and three drones.

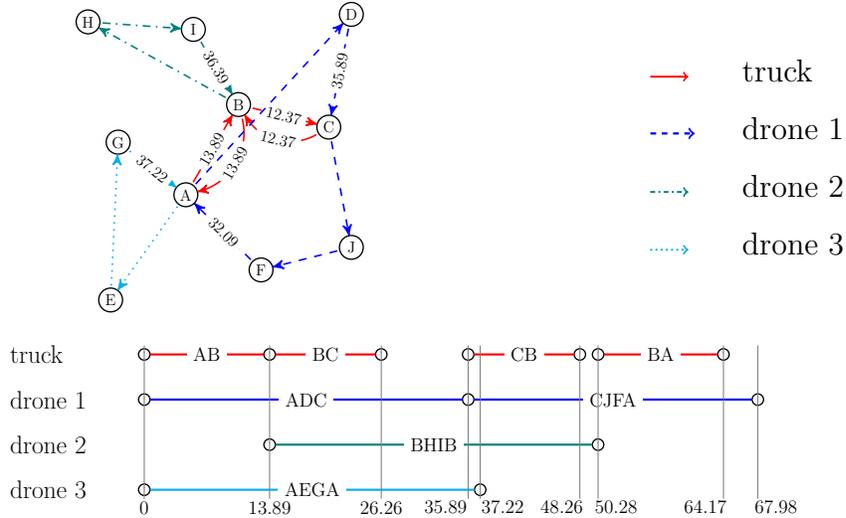
\begin{figure}
\begin{minipage}{0.7\textwidth}
\begin{center}
\begin{tikzpicture}[scale=1,transform shape]
  \tikzset{VertexStyle/.append style={scale=0.5}}
  \Vertex[x=3.0,y=4.0]{A}
  \Vertex[x=3.7,y=5.2]{B}
  \Vertex[x=4.9,y=4.9]{C}
  \Vertex[x=5.2,y=6.4]{D}
  \Vertex[x=2.0,y=2.6]{E}
  \Vertex[x=4.0,y=3.0]{F}
  \Vertex[x=2.1,y=4.7]{G}
  \Vertex[x=1.7,y=6.3]{H}
  \Vertex[x=3.1,y=6.2]{I}
  \Vertex[x=5.2,y=3.3]{J}
  
  \tikzstyle{LabelStyle}=[scale=0.5,fill=white,sloped]
  \tikzstyle{EdgeStyle}=[post, color=red]
  \Edge[label=$13.89$](A)(B)
  \Edge[label=$12.37$](B)(C)
  \tikzstyle{EdgeStyle}=[post, bend left=45, color=red]
  \Edge[label=$12.37$](C)(B)
  \Edge[label=$13.89$](B)(A)
  \tikzstyle{EdgeStyle}=[post, bend left= 45, color=blue, dashed]
  
  \tikzstyle{EdgeStyle}=[post, color=blue, dashed]
  \Edge(A)(D)
  \Edge[label=$35.89$](D)(C)
  
  \Edge(C)(J)
  \Edge(J)(F)
  \Edge[label=$32.09$](F)(A)
  
  \tikzstyle{EdgeStyle}=[post, color=cyan, dotted]
  \Edge(A)(E)
  \Edge(E)(G)
  \Edge[label=$37.22$](G)(A)
  
  \tikzstyle{EdgeStyle}=[post, color=teal, dash dot]
  \Edge(B)(H)
  \Edge(H)(I)
  \Edge[label=$36.39$](I)(B)
\end{tikzpicture}
\end{center}
\end{minipage}
\begin{minipage}{0.25\textwidth}
\begin{enumerate}
\itemsep 0.2em
\item[\textcolor{red}{$\xregarrow$}\quad] truck
\item[\textcolor{blue}{$\xdasharrow$}\quad] drone 1
\item[\textcolor{teal}{$\xdotdasharrow$}\quad] drone 2
\item[\textcolor{cyan}{$\xdottedarrow$}\quad] drone 3
\end{enumerate}
\end{minipage}
\begin{center}
\begin{tikzpicture}[scale=0.12,transform shape]

\node[font=\fontsize{80}{0}\selectfont, anchor=west] at (-15,0) {truck};
\node[font=\fontsize{80}{0}\selectfont, anchor=west] at (-15,-5) {drone 1};
\node[font=\fontsize{80}{0}\selectfont, anchor=west] at (-15,-10) {drone 2};
\node[font=\fontsize{80}{0}\selectfont, anchor=west] at (-15,-15) {drone 3};
    
  \tikzset{VertexStyle/.append style={scale=2}}
  \Vertex[L=',x=0.0,y=0.0]{A1}
  \Vertex[L=',x=0.0,y=-5.0]{A2}
  \Vertex[L=',x=0.0,y=-15.0]{A4}
  
  \Vertex[L=',x=13.89, y=0.0]{B1}
  \Vertex[L=',x=26.26, y=0.0]{C1}
  \Vertex[L=',x=35.89, y=0.0]{D1}
  \Vertex[L=',x=48.26, y=0.0]{E1}
  \Vertex[L=',x=50.28, y=0.0]{F1}
  \Vertex[L=',x=64.17, y=0.0]{G1}
  
  \Vertex[L=',x=35.89, y=-5.0]{D2}
  \Vertex[L=',x=67.98, y=-5.0]{H2}
  
  \Vertex[L=',x=13.89, y=-10.0]{B3}
  \Vertex[L=',x=50.28, y=-10.0]{F3}
  
  \Vertex[L=',x=37.22, y=-15.0]{I4}
  
  \tikzstyle{EdgeStyle}=[color=red, scale=5]
  \Edge[label=AB](A1)(B1)
  \Edge[label=BC](B1)(C1)
  \Edge[label=CB](D1)(E1)
  \Edge[label=BA](F1)(G1)
  
  \tikzstyle{EdgeStyle}=[color=blue, scale=5]
  \Edge[label=ADC](A2)(D2)
  \Edge[label=CJFA](D2)(H2)
  
  \tikzstyle{EdgeStyle}=[color=teal, scale=5]
  \Edge[label=BHIB](B3)(F3)
  
  \tikzstyle{EdgeStyle}=[color=cyan, scale=5]
  \Edge[label=AEGA](A4)(I4)
  
  \draw[gray] (0.0,1.0) -- (0.0,-16.0);
  \node[scale=5] (A) at (0.0,-17.0) {$0$};
  \draw[gray] (13.89,1.0) -- (13.89,-16.0);
  \node[scale=5] (B) at (13.89,-17.0) {$13.89$};
  \draw[gray] (26.26,1.0) -- (26.26,-16.0);
  \node[scale=5] (C) at (26.26,-17.0) {$26.26$};
  \draw[gray] (35.89,1.0) -- (35.89,-16.0);
  \node[scale=5, label={[xshift=-2.5cm, yshift=-2.0cm, scale=5]$35.89$}] (D) at (35.89,-17.0) {};
  \draw[gray] (37.22,1.0) -- (37.22,-16.0);
  \node[scale=5, label={[xshift=2.5cm, yshift=-2cm, scale=5]$37.22$}] (E) at (37.22,-17.0) {};
  \draw[gray] (48.26,1.0) -- (48.26,-16.0);
  \node[scale=5, label={[xshift=-2.0cm, yshift=-2.0cm, scale=5]$48.26$}] (F) at (48.26,-17.0) {};
  \draw[gray] (50.28,1.0) -- (50.28,-16.0);
  \node[scale=5, label={[xshift=2.0cm, yshift=-2cm, scale=5]$50.28$}] (G) at (50.28,-17.0) {};
  \draw[gray] (64.17,1.0) -- (64.17,-16.0);
  \node[scale=5, label={[xshift=-2.0cm, yshift=-2.0cm, scale=5]$64.17$}] (H) at (64.17,-17.0) {};
  \draw[gray] (67.98,1.0) -- (67.98,-16.0);
  \node[scale=5, label={[xshift=2.0cm, yshift=-2cm, scale=5]$67.98$}] (I) at (67.98,-17.0) {};
\end{tikzpicture}
\captionof{figure}{A feasible solution for an instance with 10 nodes, three drones, and proportional metrics with $\alpha = \nicefrac{4}{3}$. The routes of the truck and the drones (at the top) and the corresponding Gantt chart are shown.}
\label{fig:feas_sol}
\end{center}
\end{figure}

In line with the FSTSP variant of \cite{agatz18}, we allow the truck to visit customers multiple times. 
Note that this also opens the possibility for the truck to traverse the same arc multiple times, i.e., there may be pairs of customers $i, j$ such that the truck travels directly from customer $i$ to customer $j$ on multiple occasions throughout its tour.
We refer to solutions in which this happens as \emph{arc-retraversing}. Figure~\ref{fig:twice_sol} shows an example of an (optimal) arc-retraversing solution for an instance with five nodes and two drones. Notice that the truck can traverse the same edge $\left\lbrace i,j\right\rbrace$ twice in the two opposite directions $\left( i,j\right)$ and $\left( j,i\right)$ without retraversing the same arc. In Section~\ref{time-saving_subsection}, we describe instances where all the optimal solutions are arc-retraversing. Consequently, it is necessary to allow this type of solutions in our problem statement.

We observe that some arc-retraversing solutions are not feasible in the IP model of \cite{agatz18}. In particular, the solution space of the latter IP does not include any solutions in which the truck traverses the same arc twice when during both traversals the drone is not airborne. Indeed, the two distinct traversals of the same arc constitute identical operations and thus correspond to the same binary variable. In Section~\ref{results_subsec}, we prove that in the FSTSP model studied by \cite{agatz18}, at least one optimal solution is not arc-retraversing, and hence the formulation indeed finds an optimal solution. 
However, as soon as the FSTSP is generalized as to incorporate payloads, multiple customers per sortie, or multiple drones, excluding arc-retraversing solutions may lead to excluding all optimal solutions.

We distinguish between the TSP-mD and the two restrictions where arcs cannot be retraversed (m-CIRCUIT) and nodes cannot be revisited (m-CYCLE).
\begin{equation}
\label{relax}
\text{TSP-mD}\leq\text{m-CIRCUIT}\leq\text{m-CYCLE}\leq\text{TSP}.
\end{equation}
The definitions of these restrictions of the TSP-mD are functional to the results of Section~\ref{sec_notincluding}.

\begin{figure}
\begin{minipage}{0.7\textwidth}
\begin{center}
\begin{tikzpicture}[scale=0.85,transform shape,rotate=45]
  \tikzset{VertexStyle/.append style={scale=0.5}}
  \Vertex[x=0.0,y=0.0,L=\rotatebox{-45}{A}]{A}
  \Vertex[x=0.0,y=2.0,L=\rotatebox{-45}{B}]{B}
  \Vertex[x=2.0,y=0.0,L=\rotatebox{-45}{C}]{C}
  \Vertex[x=0.0,y=4.8,L=\rotatebox{-45}{D}]{D}
  \Vertex[x=4.8,y=0.0,L=\rotatebox{-45}{E}]{E}
  
  \tikzstyle{LabelStyle}=[scale=0.5,fill=white,sloped]
  \tikzstyle{EdgeStyle}=[post, color=red]
  \Edge[label=$14.14$](B)(C)
  
  \tikzstyle{EdgeStyle}=[post, bend left=30, color=red]
  \Edge[label=$14.14$](C)(B)
  \Edge[label=$10.00$](A)(B)
  \Edge[label=$10.00$](C)(A)
  
  \tikzstyle{EdgeStyle}=[post, bend left=45, color=red]
  \Edge[label=$14.14$](B)(C)
  
  \tikzstyle{EdgeStyle}=[post, bend left= 45, color=blue, dashed]
  \Edge[label=$21.00$](B)(D)
  \Edge(D)(B)
  
  \tikzstyle{EdgeStyle}=[post, bend left= 45, color=teal, dash dot]
  \Edge[label=$21.00$](C)(E)
  \Edge(E)(C)
\end{tikzpicture}
\end{center}
\end{minipage}
\hfill
\begin{minipage}{0.25\textwidth}
{\small
\begin{enumerate}
\item[\textcolor{red}{$\xregarrow$}\quad] truck
\item[\textcolor{blue}{$\xdasharrow$}\quad] drone 1
\item[\textcolor{teal}{$\xdotdasharrow$}\quad] drone 2
\end{enumerate}
}
\end{minipage}
\begin{center}
\begin{tikzpicture}[scale=0.12,transform shape]
\node[font=\fontsize{70}{0}\selectfont, anchor=west] at (-15,0) {truck};
\node[font=\fontsize{70}{0}\selectfont, anchor=west] at (-15,-5) {drone 1};
\node[font=\fontsize{70}{0}\selectfont, anchor=west] at (-15,-10) {drone 2};

  \tikzset{VertexStyle/.append style={scale=2}}
  \Vertex[L=',x=0.0,y=0.0]{A1}
  \Vertex[L=',x=10.0,y=0.0]{B1}
  \Vertex[L=',x=24.14,y=0.0]{C1}
  \Vertex[L=',x=38.28,y=0.0]{D1}
  \Vertex[L=',x=52.42,y=0.0]{E1}
  \Vertex[L=',x=62.42,y=0.0]{F1}
  
  \Vertex[L=',x=10.0,y=-5.0]{B2}
  \Vertex[L=',x=31.0,y=-5.0]{G2}
  
  \Vertex[L=',x=24.14,y=-10.0]{C3}
  \Vertex[L=',x=45.14,y=-10.0]{H3}
  
  \tikzstyle{EdgeStyle}=[color=red, scale=5]
  \Edge[label=AB](A1)(B1)
  \Edge[label=BC](B1)(C1)
  \Edge[label=CB](C1)(D1)
  \Edge[label=BC](D1)(E1)
  \Edge[label=CA](E1)(F1)
  
  \tikzstyle{EdgeStyle}=[color=blue, scale=5]
  \Edge[label=BDB](B2)(G2)
  
  \tikzstyle{EdgeStyle}=[color=teal, scale=5]
  \Edge[label=CEC](C3)(H3)
  
  \draw[gray] (0.0,1.0) -- (0.0,-11.0);
  \node[scale=5] (A) at (0.0,-12.0) {$0$};
  \draw[gray] (10.0,1.0) -- (10.0,-11.0);
  \node[scale=5] (B) at (10.0,-12.0) {$10.00$};
  \draw[gray] (24.14,1.0) -- (24.14,-11.0);
  \node[scale=5] (C) at (24.14,-12.0) {$24.14$};
  \draw[gray] (31.0,1.0) -- (31.0,-11.0);
  \node[scale=5] (D) at (31.0,-12.0) {$31.00$};
  \draw[gray] (38.22,1.0) -- (38.28,-11.0);
  \node[scale=5] (E) at (38.28,-12.0) {$38.28$};
  \draw[gray] (45.14,1.0) -- (45.14,-11.0);
  \node[scale=5] (F) at (45.14,-12.0) {$45.14$};
  \draw[gray] (52.42,1.0) -- (52.42,-11.0);
  \node[scale=5] (G) at (52.42,-12.0) {$52.42$};
  \draw[gray] (62.42,1.0) -- (62.42,-11.0);
  \node[scale=5] (H) at (62.42,-12.0) {$62.42$};
\end{tikzpicture}
\captionof{figure}{An optimal solution of an instance with five nodes, $L=28.0$, and proportional metrics with $\alpha=\nicefrac{4}{3}$. The truck traverses arc $\left( B,C\right)$ twice.}
\label{fig:twice_sol}
\end{center}
\end{figure}
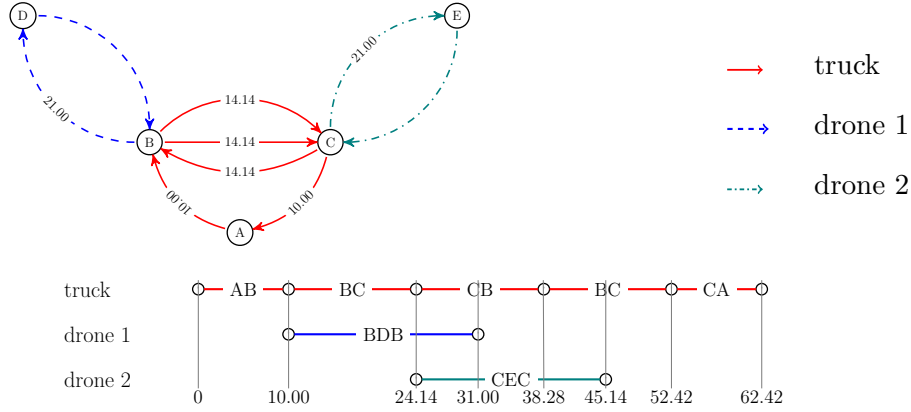

\section{Arc-retraversing solutions}
\label{sec_arcretrav}
In this section, we describe instances for which all the optimal solutions are arc-retraversing (Section~\ref{time-saving_subsection}), and provide conditions under which there always exists an optimal solution that is not arc-retraversing (Section~\ref{results_subsec}).

\subsection{Necessity of arc-retraversing solutions}
\label{time-saving_subsection}
We prove that it is necessary to include arc-retraversing solutions in the solution space of the \mbox{TSP-mD} by describing instances whose optimal solutions are all arc-retraversing. We show two instances that possess this property, with Euclidean metrics and all-zero payloads. 
The first instance has a single drone (Proposition~\ref{prop_retrav1}), while the second one satisfies $N=N^{\mathit{tr}}=N^{\mathit{dr}}$ (Proposition~\ref{prop_retrav2}). The proofs of the corresponding results require us to solve the TSP-mD a number of times.
In the Appendix~\ref{app_model} of this paper, we provide the MILP model by which we solve the relevant instances.

\begin{prop}
\label{prop_retrav1}
There exists an instance with a single drone, all-zero payloads, and Euclidean metrics such that all the optimal solutions are arc-retraversing.
\end{prop}

\begin{proof}
The proof is complete if we can find an instance with Euclidean metrics, all-zero payloads and $m =1$, and two feasible solutions $\mathcal{S}_{1}$ and $\mathcal{S}_{2}$ such that $\mathcal{S}_{1}$ is optimal under the further restriction that no arc can be retraversed, and the respective objective function values $\mathrm{OBJ}_{1}$ and $\mathrm{OBJ}_{2}$ satisfy $\mathrm{OBJ}_{1}>\mathrm{OBJ}_{2}$.

Such an instance is described in Appendix~\ref{subapp_thrice_onlyonce_sol}, and shown in Figures~\ref{fig:thrice_onlyonce_sol} and \ref{fig:thrice_sol}; these figures represent solutions $\mathcal{S}_{1}$ and $\mathcal{S}_{2}$, and their objective function values $\mathrm{OBJ}_{1}$ and $\mathrm{OBJ}_{2}$ amount to $1084.09$ and $1050.36$, respectively. These solutions were obtained by solving the MILP model in Appendix~\ref{app_model} to optimality by a commercial solver.\end{proof}

\begin{figure}
\begin{minipage}{0.8\textwidth}
\begin{center}
\begin{tikzpicture}[scale=0.25,transform shape]
  \tikzset{VertexStyle/.append style={scale=1.67}}
  \Vertex[x=0.0,y=0.0,L={A}]{A}
  \Vertex[x=-1.0,y=3.0,L={B}]{B}
  \Vertex[x=1.0,y=3.0,L={C}]{C}
  \Vertex[x=-6.0,y=5.0,L={D}]{D}
  \Vertex[x=6.0,y=5.0,L={E}]{E}
  \Vertex[x=-18.0,y=10.0,L={F}]{F}
  \Vertex[x=-5.0,y=10.0,L={G}]{G}
  \Vertex[x=5.0,y=10.0,L={H}]{H}
  \Vertex[x=18.0,y=10.0,L={I}]{I}
  \Vertex[x=-3.0,y=14.0,L={J}]{J}
  \Vertex[x=3.0,y=14.0,L={K}]{K}
  \Vertex[x=-6.0,y=15.0,L={L}]{L}
  \Vertex[x=6.0,y=15.0,L={M}]{M}
  \Vertex[x=-1.0,y=17.0,L={N}]{N}
  \Vertex[x=1.0,y=17.0,L={O}]{O}
  
  \tikzstyle{LabelStyle}=[scale=1.67,fill=white,sloped]
  \tikzstyle{EdgeStyle}=[post, color=red]
  \Edge[label=$10.00$](H)(G)
  \Edge[label=$200.00$](I)(H)
  
  \tikzstyle{EdgeStyle}=[post, bend left=60, color=red]
  \Edge[label=$210.00$](H)(F)
  
  \tikzstyle{EdgeStyle}=[post, bend left=20, color=red]
  \Edge[label=$10.00$](G)(H)
  
  \tikzstyle{EdgeStyle}=[post, bend right=30, color=red]
  \Edge[label=$228.09$](A)(I)
  \tikzstyle{EdgeStyle}=[post, bend right=35, color=red]
  \Edge[label=$228.09$](F)(A)
  
  \tikzstyle{EdgeStyle}=[post, color=blue, dashed]
  \Edge[label=$43.04$](G)(B)
  \Edge(B)(C)
  \Edge(C)(H)
  
  \Edge(O)(N)
  
  \Edge[label=$33.90$](H)(K)
  \Edge(K)(M)
  \Edge(M)(H)
  
  \Edge[label=$33.90$](G)(J)
  \Edge(J)(L)
  \Edge(L)(G)
  
  \tikzstyle{EdgeStyle}=[post, bend left=35, color=blue, dashed]
  \Edge[label=$43.04$](H)(O)
  \Edge(N)(G)
  
  \tikzstyle{EdgeStyle}=[post, bend right=20, color=blue, dashed]
  \Edge[label=$32.02$](G)(D)
  \Edge(D)(G)
  
  \tikzstyle{EdgeStyle}=[post, bend left=20, color=blue, dashed]
  \Edge[label=$32.02$](H)(E)
  \Edge(E)(H)
\end{tikzpicture}
\end{center}
\end{minipage}
\hfill
\begin{minipage}{0.15\textwidth}
{\small
\begin{enumerate}
\item[\textcolor{red}{$\xregarrow$}\quad] truck
\item[\textcolor{blue}{$\xdasharrow$}\quad] drone
\end{enumerate}
}
\end{minipage}
\begin{center}
\begin{tikzpicture}[scale=0.15,transform shape]
\node[font=\fontsize{60}{0}\selectfont, anchor=west] at (-7,0) {truck};
\node[font=\fontsize{60}{0}\selectfont, anchor=west] at (-7,-5.0) {drone};

  \tikzset{VertexStyle/.append style={scale=1}}
  \Vertex[L=',x=0.0,y=0.0]{A1}
  \Vertex[L=',x=5.0,y=0.0]{B1}
  \Vertex[L=',x=10.0,y=0.0]{C1}
  \Vertex[L=',x=15.0,y=0.0]{D1}
  \Vertex[L=',x=20.0,y=0.0]{E1}
  \Vertex[L=',x=35.0,y=0.0]{H1}
  \Vertex[L=',x=40.0,y=0.0]{I1}
  \Vertex[L=',x=50.0,y=0.0]{K1}
  \Vertex[L=',x=55.0,y=0.0]{L1}
  \Vertex[L=',x=60.0,y=0.0]{M1}
  
  \Vertex[L=',x=10.0,y=-5.0]{C2}
  \Vertex[L=',x=15.0,y=-5.0]{D2}
  \Vertex[L=',x=20.0,y=-5.0]{E2}
  \Vertex[L=',x=25.0,y=-5.0]{F2}
  \Vertex[L=',x=30.0,y=-5.0]{G2}
  \Vertex[L=',x=35.0,y=-5.0]{H2}
  \Vertex[L=',x=40.0,y=-5.0]{I2}
  \Vertex[L=',x=45.0,y=-5.0]{J2}
  \Vertex[L=',x=50.0,y=-5.0]{K2}
  
  \tikzstyle{EdgeStyle}=[color=red, scale=3]
  \Edge[label=AI](A1)(B1)
  \Edge[label=IH](B1)(C1)
  \Edge[label=HG](D1)(E1)
  \Edge[label=GH](H1)(I1)
  \Edge[label=HF](K1)(L1)
  \Edge[label=FA](L1)(M1)
  
  \tikzstyle{EdgeStyle}=[color=blue, scale=3]
  \Edge[label=HKMH](C2)(D2)
  \Edge[label=HONG](D2)(F2)
  \Edge[label=GJLG](F2)(G2)
  \Edge[label=GDG](G2)(H2)
  \Edge[label=GBCH](H2)(J2)
  \Edge[label=HEH](J2)(K2)
  
  \draw[gray] (0.0,1.0) -- (0.0,-6.0);
  \node[scale=5] (A) at (0.0,-9.5) {\small $0.00$};
  \draw[gray] (5.0,1.0) -- (5.0,-6.0);
  \node[scale=5] (B) at (5.0,2.0) {\small $228.09$};
  \draw[gray] (10.0,1.0) -- (10.0,-6.0);
  \node[scale=5] (C) at (10.0,-9.5) {\small $428.09$};
  \draw[gray] (15.0,1.0) -- (15.0,-6.0);
  \node[scale=5] (D) at (15.0,2.0) {\small $461.99$};
  \draw[gray] (20.0,1.0) -- (20.0,-6.0);
  \node[scale=5] (E) at (20.0,-9.5) {\small $471.99$};
  \draw[gray] (25.0,1.0) -- (25.0,-6.0);
  \node[scale=5] (F) at (25.0,2.0) {\small $505.03$};
  \draw[gray] (30.0,1.0) -- (30.0,-6.0);
  \node[scale=5] (G) at (30.0,-9.5) {\small $538.93$};
  \draw[gray] (35.0,1.0) -- (35.0,-6.0);
  \node[scale=5] (H) at (35.0,2.0) {\small $570.95$};
  \draw[gray] (40.0,1.0) -- (40.0,-6.0);
  \node[scale=5] (I) at (40.0,-9.5) {\small $580.95$};
  \draw[gray] (45.0,1.0) -- (45.0,-6.0);
  \node[scale=5] (J) at (45.0,2.0) {\small $613.99$};
  \draw[gray] (50.0,1.0) -- (50.0,-6.0);
  \node[scale=5] (K) at (50.0,-9.5) {\small $646.01$};
  \draw[gray] (55.0,1.0) -- (55.0,-6.0);
  \node[scale=5] (L) at (55.0,2.0) {\small $856.01$};
  \draw[gray] (60.0,1.0) -- (60.0,-6.0);
  \node[scale=5] (M) at (60.0,-9.5) {\small $1084.09$};
\end{tikzpicture}
\captionof{figure}{An optimal solution of an instance with $15$ nodes, proportional metrics with $\alpha=1$, and $L=43.04$, when the truck is further constrained to traverse an arc at most once. The truck and the drone can visit the nodes in $N^{\mathit{tr}}=\left\lbrace A,F,G,H, I\right\rbrace$ and $N^{\mathit{dr}}=\left\lbrace B,C,D,E,J,K,L,M,N,O\right\rbrace$, respectively.}
\label{fig:thrice_onlyonce_sol}
\end{center}
\end{figure}
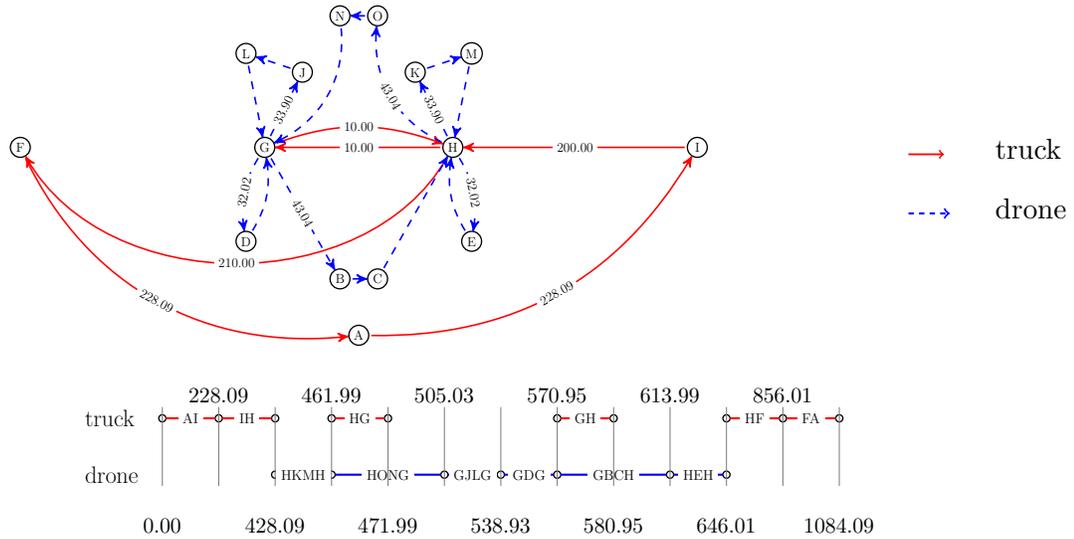

\begin{figure}
\begin{minipage}{0.8\textwidth}
\begin{center}
\begin{tikzpicture}[scale=0.25,transform shape]
  \tikzset{VertexStyle/.append style={scale=1.67}}
  \Vertex[x=0.0,y=0.0,L={A}]{A}
  \Vertex[x=-1.0,y=3.0,L={B}]{B}
  \Vertex[x=1.0,y=3.0,L={C}]{C}
  \Vertex[x=-6.0,y=5.0,L={D}]{D}
  \Vertex[x=6.0,y=5.0,L={E}]{E}
  \Vertex[x=-18.0,y=10.0,L={F}]{F}
  \Vertex[x=-5.0,y=10.0,L={G}]{G}
  \Vertex[x=5.0,y=10.0,L={H}]{H}
  \Vertex[x=18.0,y=10.0,L={I}]{I}
  \Vertex[x=-3.0,y=14.0,L={J}]{J}
  \Vertex[x=3.0,y=14.0,L={K}]{K}
  \Vertex[x=-6.0,y=15.0,L={L}]{L}
  \Vertex[x=6.0,y=15.0,L={M}]{M}
  \Vertex[x=-1.0,y=17.0,L={N}]{N}
  \Vertex[x=1.0,y=17.0,L={O}]{O}
  
  \tikzstyle{LabelStyle}=[scale=1.67,fill=white,sloped]
  \tikzstyle{EdgeStyle}=[post, color=red]
  \Edge[label=$200.00$](F)(G)
  \Edge[label=$10.00$](H)(G)
  \Edge[label=$200.00$](H)(I)
  
  \tikzstyle{EdgeStyle}=[post, bend left=20, color=red]
  \Edge[label=$10.00$](G)(H)
  
  \tikzstyle{EdgeStyle}=[post, bend right=20, color=red]
  \Edge[label=$10.00$](G)(H)
  
  \tikzstyle{EdgeStyle}=[post, bend right=60, color=red]
  \Edge[label=$10.00$](G)(H)
  
  \tikzstyle{EdgeStyle}=[post, bend left=40, color=red]
  \Edge[label=$10.00$](H)(G)
  
  \tikzstyle{EdgeStyle}=[post, bend left=30, color=red]
  \Edge[label=$228.09$](A)(F)
  \Edge[label=$228.09$](I)(A)
  
  \tikzstyle{EdgeStyle}=[post, color=blue, dashed]
  \Edge[label=$43.04$](G)(B)
  \Edge(B)(C)
  \Edge(C)(H)
  
  \Edge[label=$22.06$](H)(K)
  \Edge(K)(J)
  \Edge(J)(G)
  
  \Edge[label=$43.02$](G)(L)
  \Edge(M)(H)
  
  \Edge[label=$43.02$](H)(E)
  \Edge(D)(G)
  
  \Edge(N)(O)
  
  \tikzstyle{EdgeStyle}=[post, bend left=30, color=blue, dashed]
  \Edge[label=$43.04$](G)(N)
  \Edge(O)(H)
  
  \tikzstyle{EdgeStyle}=[post, bend left=50, color=blue, dashed]
  \Edge(L)(M)
  \Edge(E)(D)
\end{tikzpicture}
\end{center}
\end{minipage}
\hfill
\begin{minipage}{0.15\textwidth}
{\small
\begin{enumerate}
\item[\textcolor{red}{$\xregarrow$}\quad] truck
\item[\textcolor{blue}{$\xdasharrow$}\quad] drone
\end{enumerate}
}
\end{minipage}
\begin{center}
\begin{tikzpicture}[scale=0.15,transform shape]
\node[font=\fontsize{60}{0}\selectfont, anchor=west] at (-7,0) {truck};
\node[font=\fontsize{60}{0}\selectfont, anchor=west] at (-7,-5.0) {drone};

  \tikzset{VertexStyle/.append style={scale=1}}
  \Vertex[L=',x=0.0,y=0.0]{A1}
  \Vertex[L=',x=5.0,y=0.0]{B1}
  \Vertex[L=',x=10.0,y=0.0]{C1}
  \Vertex[L=',x=15.0,y=0.0]{D1}
  \Vertex[L=',x=20.0,y=0.0]{E1}
  \Vertex[L=',x=25.0,y=0.0]{F1}
  \Vertex[L=',x=30.0,y=0.0]{G1}
  \Vertex[L=',x=35.0,y=0.0]{H1}
  \Vertex[L=',x=40.0,y=0.0]{I1}
  \Vertex[L=',x=45.0,y=0.0]{J1}
  \Vertex[L=',x=50.0,y=0.0]{K1}
  \Vertex[L=',x=55.0,y=0.0]{L1}
  \Vertex[L=',x=60.0,y=0.0]{M1}
  \Vertex[L=',x=65.0,y=0.0]{N1}
  \Vertex[L=',x=70.0,y=0.0]{O1}
  
  \Vertex[L=',x=10.0,y=-5.0]{C2}
  \Vertex[L=',x=15.0,y=-5.0]{D2}
  \Vertex[L=',x=20.0,y=-5.0]{E2}
  \Vertex[L=',x=25.0,y=-5.0]{F2}
  \Vertex[L=',x=30.0,y=-5.0]{G2}
  \Vertex[L=',x=35.0,y=-5.0]{H2}
  \Vertex[L=',x=40.0,y=-5.0]{I2}
  \Vertex[L=',x=45.0,y=-5.0]{J2}
  \Vertex[L=',x=50.0,y=-5.0]{K2}
  \Vertex[L=',x=55.0,y=-5.0]{L2}
  \Vertex[L=',x=60.0,y=-5.0]{M2}
  
  \tikzstyle{EdgeStyle}=[color=red, scale=3]
  \Edge[label=AF](A1)(B1)
  \Edge[label=FG](B1)(C1)
  \Edge[label=GH](C1)(D1)
  \Edge[label=HG](E1)(F1)
  \Edge[label=GH](G1)(H1)
  \Edge[label=HG](I1)(J1)
  \Edge[label=GH](K1)(L1)
  \Edge[label=HI](M1)(N1)
  \Edge[label=IA](N1)(O1)
  
  \tikzstyle{EdgeStyle}=[color=blue, scale=3]
  \Edge[label=GBCH](C2)(E2)
  \Edge[label=HKJG](E2)(G2)
  \Edge[label=GLMH](G2)(I2)
  \Edge[label=HEDG](I2)(K2)
  \Edge[label=GNOH](K2)(M2)
  
  \draw[gray] (0.0,1.0) -- (0.0,-6.0);
  \node[scale=5] (A) at (0.0,-9.5) {\small $0.00$};
  \draw[gray] (5.0,1.0) -- (5.0,-6.0);
  \node[scale=5] (B) at (5.0,2.0) {\small $228.09$};
  \draw[gray] (10.0,1.0) -- (10.0,-6.0);
  \node[scale=5] (C) at (10.0,-9.5) {\small $428.09$};
  \draw[gray] (15.0,1.0) -- (15.0,-6.0);
  \node[scale=5] (D) at (15.0,2.0) {\small $438.09$};
  \draw[gray] (20.0,1.0) -- (20.0,-6.0);
  \node[scale=5] (E) at (20.0,-9.5) {\small $471.13$};
  \draw[gray] (25.0,1.0) -- (25.0,-6.0);
  \node[scale=5] (F) at (25.0,2.0) {\small $481.13$};
  \draw[gray] (30.0,1.0) -- (30.0,-6.0);
  \node[scale=5] (G) at (30.0,-9.5) {\small $493.19$};
  \draw[gray] (35.0,1.0) -- (35.0,-6.0);
  \node[scale=5] (H) at (35.0,2.0) {\small $503.19$};
  \draw[gray] (40.0,1.0) -- (40.0,-6.0);
  \node[scale=5] (I) at (40.0,-9.5) {\small $536.21$};
  \draw[gray] (45.0,1.0) -- (45.0,-6.0);
  \node[scale=5] (J) at (45.0,2.0) {\small $546.21$};
  \draw[gray] (50.0,1.0) -- (50.0,-6.0);
  \node[scale=5] (K) at (50.0,-9.5) {\small $579.23$};
  \draw[gray] (55.0,1.0) -- (55.0,-6.0);
  \node[scale=5] (L) at (55.0,2.0) {\small $589.23$};
  \draw[gray] (60.0,1.0) -- (60.0,-6.0);
  \node[scale=5] (M) at (60.0,-9.5) {\small $622.27$};
  \draw[gray] (65.0,1.0) -- (65.0,-6.0);
  \node[scale=5] (N) at (65.0,2.0) {\small $822.27$};
  \draw[gray] (70.0,1.0) -- (70.0,-6.0);
  \node[scale=5] (O) at (70.0,-9.5) {\small $1050.36$};
\end{tikzpicture}
\captionof{figure}{A feasible solution of the instance shown in Figure~\ref{fig:thrice_onlyonce_sol}, without any constraints on the number of traversals of any arc.}
\label{fig:thrice_sol}
\end{center}
\end{figure}
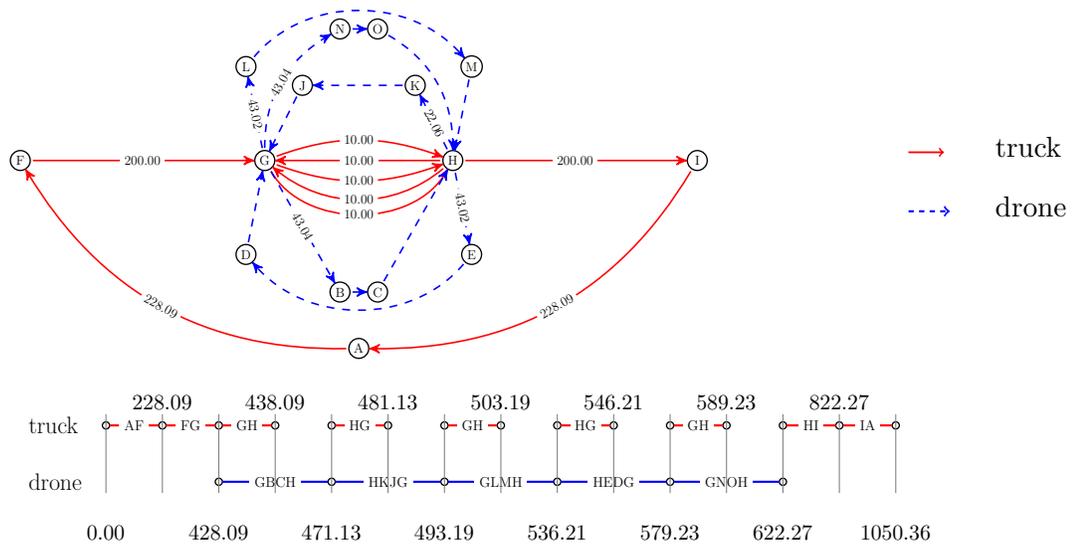

Notice that, in the instance shown in Figure~\ref{fig:thrice_onlyonce_sol}, the truck and the drones are not allowed to visit all the nodes, i.e., $N^{\mathit{tr}}\neq N$ and $N^{\mathit{dr}}\neq N$. An analogous result to Proposition~\ref{prop_retrav1} also holds when $N=N^{\mathit{dr}}=N^{\mathit{tr}}$, even when the feasible sorties serve only one customer.
\begin{prop}
\label{prop_retrav2}
There exists an instance with Euclidean metrics and all-zero payloads such that $N=N^{\mathit{dr}}=N^{\mathit{tr}}$, the feasible sorties only serve one customer, and all the optimal solutions are arc-retraversing.
\end{prop}

\begin{proof}We follow an analogous argument to that of Proposition~\ref{prop_retrav1}. Consider the instance described in Appendix~\ref{subapp_twice_sol} and shown in Figures \ref{fig:twice_sol} and \ref{fig:once_sol}. On the one hand, the solution shown in Figure~\ref{fig:once_sol} is optimal under the condition that no arc-retraversing solutions are allowed, and leads to a completion time of $76.14$; on the other hand, that of Figure~\ref{fig:twice_sol} (which is optimal and retraverses one arc) leads to a strictly lower completion time, namely, $62.42$. \end{proof}

For the instance shown in Figures \ref{fig:twice_sol} and \ref{fig:once_sol}, not allowing arc-retraversing solutions leads to an increase of $18\%$ of the optimal completion time.

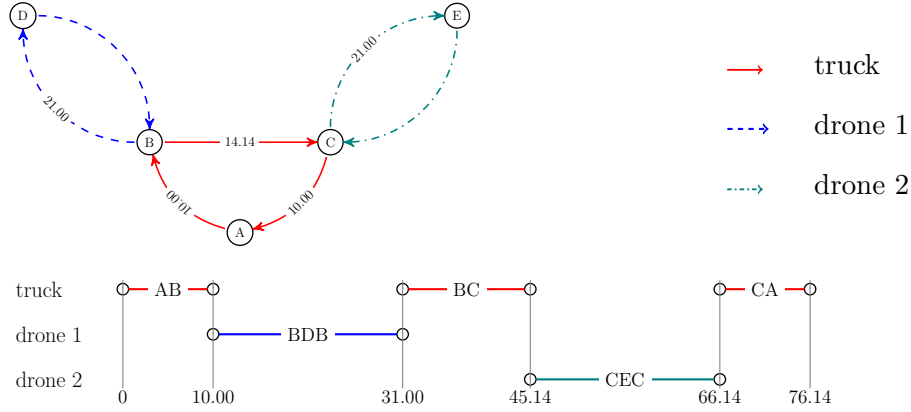
\begin{figure}
\begin{minipage}{0.7\textwidth}
\begin{center}
\begin{tikzpicture}[scale=0.85,transform shape,rotate=45]
  \tikzset{VertexStyle/.append style={scale=0.5}}
  \Vertex[x=0.0,y=0.0,L=\rotatebox{-45}{A}]{A}
  \Vertex[x=0.0,y=2.0,L=\rotatebox{-45}{B}]{B}
  \Vertex[x=2.0,y=0.0,L=\rotatebox{-45}{C}]{C}
  \Vertex[x=0.0,y=4.8,L=\rotatebox{-45}{D}]{D}
  \Vertex[x=4.8,y=0.0,L=\rotatebox{-45}{E}]{E}
  
  \tikzstyle{LabelStyle}=[scale=0.5,fill=white,sloped]
  \tikzstyle{EdgeStyle}=[post, color=red]
  \Edge[label=$14.14$](B)(C)
  
  \tikzstyle{EdgeStyle}=[post, bend left=30, color=red]
  \Edge[label=$10.00$](A)(B)
  \Edge[label=$10.00$](C)(A)
  
  \tikzstyle{EdgeStyle}=[post, bend left=45, color=red]
  
  \tikzstyle{EdgeStyle}=[post, bend left= 45, color=blue, dashed]
  \Edge[label=$21.00$](B)(D)
  \Edge(D)(B)
  
  \tikzstyle{EdgeStyle}=[post, bend left= 45, color=teal, dash dot]
  \Edge[label=$21.00$](C)(E)
  \Edge(E)(C)
\end{tikzpicture}
\end{center}
\end{minipage}
\hfill
\begin{minipage}{0.25\textwidth}
{\small
\begin{enumerate}
\item[\textcolor{red}{$\xregarrow$}\quad] truck
\item[\textcolor{blue}{$\xdasharrow$}\quad] drone 1
\item[\textcolor{teal}{$\xdotdasharrow$}\quad] drone 2
\end{enumerate}
}
\end{minipage}
\begin{center}
\begin{tikzpicture}[scale=0.12,transform shape]
\node[font=\fontsize{70}{0}\selectfont, anchor=west] at (-12,0) {truck};
\node[font=\fontsize{70}{0}\selectfont, anchor=west] at (-12,-5) {drone 1};
\node[font=\fontsize{70}{0}\selectfont, anchor=west] at (-12,-10) {drone 2};

  \tikzset{VertexStyle/.append style={scale=2}}
  \Vertex[L=',x=0.0,y=0.0]{A1}
  \Vertex[L=',x=10.0,y=0.0]{B1}
  \Vertex[L=',x=31.0,y=0.0]{G1}
  \Vertex[L=',x=45.14,y=0.0]{H1}
  \Vertex[L=',x=66.14,y=0.0]{I1}
  \Vertex[L=',x=76.14,y=0.0]{J1}
  
  \Vertex[L=',x=10.0,y=-5.0]{B2}
  \Vertex[L=',x=31.0,y=-5.0]{G2}
  
  \Vertex[L=',x=45.14,y=-10.0]{H3}
  \Vertex[L=',x=66.14,y=-10.0]{I3}
  
  \tikzstyle{EdgeStyle}=[color=red, scale=5]
  \Edge[label=AB](A1)(B1)
  \Edge[label=BC](G1)(H1)
  \Edge[label=CA](I1)(J1)
  
  \tikzstyle{EdgeStyle}=[color=blue, scale=5]
  \Edge[label=BDB](B2)(G2)
  
  \tikzstyle{EdgeStyle}=[color=teal, scale=5]
  \Edge[label=CEC](H3)(I3)
  
  \draw[gray] (0.0,1.0) -- (0.0,-11.0);
  \node[scale=5] (A) at (0.0,-12.0) {$0$};
  \draw[gray] (10.0,1.0) -- (10.0,-11.0);
  \node[scale=5] (B) at (10.0,-12.0) {$10.00$};
  \draw[gray] (31.0,1.0) -- (31.0,-11.0);
  \node[scale=5] (D) at (31.0,-12.0) {$31.00$};
  \draw[gray] (45.14,1.0) -- (45.14,-11.0);
  \node[scale=5] (F) at (45.14,-12.0) {$45.14$};
  \draw[gray] (66.14,1.0) -- (66.14,-11.0);
  \node[scale=5] (I) at (66.14,-12.0) {$66.14$};
  \draw[gray] (76.14,1.0) -- (76.14,-11.0);
  \node[scale=5] (J) at (76.14,-12.0) {$76.14$};
\end{tikzpicture}
\captionof{figure}{An optimal solution of the instance shown in Figure~\ref{fig:twice_sol}, when the truck is constrained not to traverse any arc more than once.}
\label{fig:once_sol}
\end{center}
\end{figure}

\subsection{Sufficiency of non-arc-retraversing solutions}
\label{results_subsec}
In this section, we provide conditions under which there exists an optimal solution that is not arc-retraversing.

\begin{lem}
\label{lemma1}
There exists an optimal solution such that if a node $i\in N\smallsetminus\left\lbrace 0\right\rbrace$ is visited multiple times, then a drone is launched or retrieved every time the truck visits $i$.
\end{lem}
\begin{proof}Suppose by contradiction that a node $i\in N\smallsetminus\left\lbrace 0\right\rbrace$ is visited multiple times, but during one of these times no drone is launched nor retrieved. Then, we can shortcut the truck route at $i$, without missing any drone operations, and 
the resulting route will not be strictly longer than the original one.\end{proof}

An analogous result to Lemma~\ref{lemma1} applies to the depot as well, starting from its third visit on (the first two visits being the start and the end of the whole operation). Given a feasible solution represented by a truck route $\pi^{\mathit{tr}}$ and drone sorties, we define the \emph{support} of a sortie $\pi^{\mathit{dr}}$ as the sub-route $\pi'\subseteq \pi^{\mathit{tr}}$ traversed by the truck when a drone flies along $\pi^{\mathit{dr}}$. For example, in Figure~\ref{fig:pinned}, the support of the sorties $\left( B,G,D\right)$, $\left( D,H,D\right)$, $\left( D,I,F\right)$, $\left( F,J,B\right)$ are $\left( B,C,D\right)$, $D$, $\left( D,E,F\right)$, and the arc $\left( F,B\right)$, respectively.

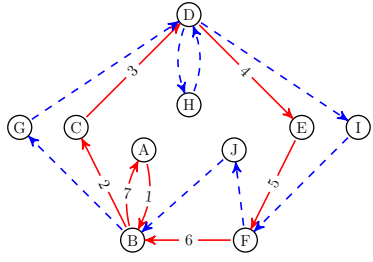
\begin{figure}
\begin{center}
\begin{tikzpicture}[scale=0.3,transform shape]
  \tikzset{VertexStyle/.append style={scale=1.67}}
  \Vertex[x=-2.0,y=9.0,L={A}]{A}
  \Vertex[x=-2.5,y=5.0,L={B}]{B}
  \Vertex[x=-5.0,y=10.0,L={C}]{C}
  \Vertex[x=0.0,y=15.0,L={D}]{D}
  \Vertex[x=5.0,y=10.0,L={E}]{E}
  \Vertex[x=2.5,y=5.0,L={F}]{F}
  
  \Vertex[x=-7.5,y=10.0,L={G}]{G}
  \Vertex[x=0.0,y=11.0,L={H}]{H}
  \Vertex[x=7.5,y=10.0,L={I}]{I}
  \Vertex[x=2.0,y=9.0,L={J}]{J}
  
  \tikzstyle{LabelStyle}=[scale=1.67,fill=white,sloped]
  \tikzstyle{EdgeStyle}=[post, color=red]
  \Edge[label=$2$](B)(C)
  \Edge[label=$3$](C)(D)
  \Edge[label=$4$](D)(E)
  \Edge[label=$5$](E)(F)
  \Edge[label=$6$](F)(B)
  
  \tikzstyle{EdgeStyle}=[post, color=red, bend left=20]
  \Edge[label=\rotatebox{-90}{$1$}](A)(B)
  \Edge[label=\rotatebox{-90}{$7$}](B)(A)
  
  \tikzstyle{EdgeStyle}=[post, color=blue, dashed]
  \Edge(B)(G)
  \Edge(G)(D)
  \Edge(D)(I)
  \Edge(I)(F)
  \Edge(F)(J)
  \Edge(J)(B)
  
  \tikzstyle{EdgeStyle}=[post, color=blue, dashed, bend right=20]
  \Edge(D)(H)
  \Edge(H)(D)
\end{tikzpicture}
\captionof{figure}{Illustration of the definition of support of a sortie. The numbers indicate the positions of the corresponding arcs in the truck route.}
\label{fig:pinned}
\end{center}
\end{figure}

\begin{lem}
\label{lemma2}
Let $m=1$. There exists an optimal solution such that if an arc $\left( i,j\right)\in A$ is traversed multiple times, then the following conditions hold:
\begin{itemize}
    \item[\emph{i.}] if the support of a sortie $\pi'$ contains $\left( i,j\right)$, then the support of $\pi'$ is equal to arc $\left( i,j\right)$;
    \item[\emph{ii.}] if the support of a sortie $\pi'$ shares at least one arc with a truck sub-route $\pi_{ji}$ from $j$ to $i$ between two consecutive traversals of $\left( i,j\right)$, then the support of $\pi'$ is contained in $\pi_{ji}$.
\end{itemize}
\end{lem}
\begin{proof} If the support of a sortie $\pi'$ contains $\left( i,j\right)$ and at least one more arc, then the truck visits either $i$ or $j$ without launching or retrieving the drone. Then, by Lemma~\ref{lemma1}, we can shortcut its route at the redundant node without losing optimality.

Suppose that a sortie $\pi'$ shares at least one arc with a truck sub-route $\pi_{ji}$ like in the hypothesis. Without loss of generality, we can assume that the arcs shared by both $\pi'$ and $\pi_{ji}$ are only traversed once. Indeed, if not, then the thesis holds by step~(i.). If the support of $\pi'$ contains an arc that does not belong to $\pi_{ji}$, then the truck must visit either $i$ or $j$ while the drone is airborne. Again, by Lemma~\ref{lemma1}, we can shortcut the truck route at the redundant node without losing optimality.
\end{proof}

Figure~\ref{fig:arc_pinned} illustrates the thesis of Lemma~\ref{lemma2}. In this example, the support of sorties $\left( B,G,C\right)$ and $\left( B,F,C\right)$ is the arc $\left( B,C\right)$, and the supports of sorties $\left( C,H,D\right)$ and $\left( D,I,B\right)$ are the truck sub-routes $\left( C,D\right)$ and $\left( D,E,B\right)$, respectively.

\begin{figure}
\centering
\begin{tikzpicture}[scale=0.3,transform shape]
  \tikzset{VertexStyle/.append style={scale=1.67}}
  \Vertex[x=-5.0,y=5.0,L={A}]{A}
  \Vertex[x=0.0,y=0.0,L={B}]{B}
  \Vertex[x=0.0,y=10.0,L={C}]{C}
  \Vertex[x=6.5,y=7.5,L={D}]{D}
  \Vertex[x=6.5,y=2.5,L={E}]{E}
  
  \Vertex[x=2.5,y=5.0,L={F}]{F}
  \Vertex[x=-2.5,y=5.0,L={G}]{G}
  \Vertex[x=4.0,y=7.0,L={H}]{H}
  \Vertex[x=4.0,y=3.0,L={I}]{I}
  
  \tikzstyle{LabelStyle}=[scale=1.67,fill=white,sloped]
  \tikzstyle{EdgeStyle}=[post, color=red]
  
  \Edge[label=$3$](C)(D)
  \Edge[label=$4$](D)(E)
  \Edge[label=$5$](E)(B)

  \tikzstyle{EdgeStyle}=[post, color=red, bend left=20]
  \Edge[label=\rotatebox{90}{$2$}](B)(C)
  \tikzstyle{EdgeStyle}=[post, color=red, bend right=20]
  \Edge[label=\rotatebox{-90}{$6$}](B)(C)
  \Edge[label=$1$](A)(B)
  \Edge[label=$7$](C)(A)
  
  \tikzstyle{EdgeStyle}=[post, color=blue, dashed]
  \Edge(B)(G)
  \Edge(G)(C)
  \Edge(C)(H)
  \Edge(H)(D)
  \Edge(D)(I)
  \Edge(I)(B)
  \Edge(B)(F)
  \Edge(F)(C)
\end{tikzpicture}
\caption{Illustration of the thesis of Lemma~\ref{lemma2}.}
\label{fig:arc_pinned}
\end{figure}
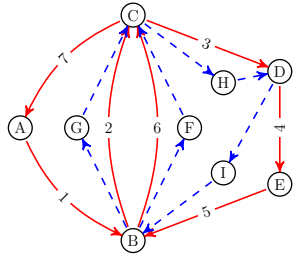

\begin{prop}
\label{saving_prop1}
At least one optimal solution is not arc-retraversing if only a single drone is available, all sorties are invertible and serve at most one customer.
\end{prop}
\begin{proof} 

Consider an optimal solution with the properties described in Lemma~\ref{lemma2}, and denote its truck route by $\pi$. Without loss of generality, suppose that at least one arc in $\pi$ is traversed multiple times (if not, there is nothing to prove): we choose one and denote it by $\left( i,j\right)$. The proof is complete if we show that there always exists another truck route $\pi'$ that does not lead to a (strictly) longer completion time, and that traverses (with multiplicity) two arcs less than $\pi$.

Because the truck travels from $i$ to $j$ at least twice, its route must contain a path from $j$ to $i$, between two consecutive traversals of $\left( i,j\right)$, that we denote by $\pi_{ji}$. Then, we call $\pi_{0i}$ and $\pi_{j0}$ the paths from the depot to $i$ and from $j$ to the depot (respectively) such that $\pi = \pi_{0i}\circ \left( i,j\right) \circ \pi_{ji} \circ \left( i,j\right) \circ \pi_{j0}$. Since we only have one drone, Lemma~\ref{lemma2} implies that the support of all the sorties performed while the truck travels along $\left( i,j\right)$ or $\pi_{ji}$ are contained in $\left( i,j\right)$ and $\pi_{ji}$, respectively. Because sorties are invertible by hypothesis, we can invert both the truck route $\pi_{ji}$ and the sorties whose support is contained in it; we call the resulting inverted route $\pi_{ji}^{-1}$. The new truck route defined by $\pi_{0i}\circ \pi_{ji}^{-1} \circ \pi_{j0}$ does not miss any drone operations whose supports were originally contained in $\pi_{0i}$, $\pi_{ji}$ or $\pi_{j0}$; however, by defining the arc $\left( i,j\right)$ out of the new route twice, we might miss all those sorties whose supports were originally the arc $\left( i,j\right)$ (at most two of them, since we removed $\left( i,j\right)$ twice, and only a single drone is available).

Because we are only allowed to visit one customer per sortie, such sorties must be of the form $\left( i,k,j\right)$, for some intermediate customer $k\in N^{\mathit{dr}}\smallsetminus\left\lbrace i,j\right\rbrace$. If $\ell'_{ik}\leq\ell'_{kj}$, then at node $i$, as soon as the truck has traveled along $\pi_{0i}$ and would otherwise be ready to leave $i$, we can instead launch the drone back and forth from $i$ to $k$, and let the truck wait for it at $i$. Otherwise, we can launch the drone back and forth from $j$ to $k$, immediately after the truck has traveled along $\pi_{ji}^{-1}$. These (possibly two) new sorties are by construction not longer than the original sorties of the form $\left( i,k,j\right)$. Hence, they are feasible and, when incorporated into the new truck route $\pi' = \pi_{0i}\circ\pi_{ji}^{-1}\circ\pi_{j0}$, they lead to a completion time that is not greater than $\pi$'s original one. In this solution, the truck traverses (with multiplicity) two arcs less than $\pi$.\end{proof}

\begin{figure}
\begin{center}
\begin{minipage}{0.45\textwidth}
\centering
\begin{tikzpicture}[scale=0.3,transform shape]
  \tikzset{VertexStyle/.append style={scale=1.67}}
  \Vertex[x=0.0,y=0.0,L={A}]{A}
  \Vertex[x=-5.0,y=6.0,L={B}]{B}
  \Vertex[x=5.0,y=6.0,L={C}]{C}
  \Vertex[x=0.0,y=3.0,L={D}]{D}
  \Vertex[x=0.0,y=9.0,L={E}]{E}
  \Vertex[x=0.0,y=12.0,L={F}]{F}
  
  \tikzstyle{LabelStyle}=[scale=1.67,fill=white,sloped]
  \tikzstyle{EdgeStyle}=[post, color=red]
  \Edge[label=$1$](A)(B)
  \Edge[label=$3$](C)(F)
  \Edge[label=$4$](F)(B)
  \Edge[label=$6$](C)(A)
  
  \tikzstyle{EdgeStyle}=[post, color=red, bend left=20]
  \Edge[label=$2$](B)(C)
  \tikzstyle{EdgeStyle}=[post, color=red, bend right=20]
  \Edge[label=$5$](B)(C)
  
  \tikzstyle{EdgeStyle}=[post, color=blue, dashed]
  \Edge(B)(D)
  \Edge(D)(C)
  \Edge(B)(E)
  \Edge(E)(C)
\end{tikzpicture}
\end{minipage}
\begin{minipage}{0.45\textwidth}
\centering
\begin{tikzpicture}[scale=0.3,transform shape]
  \tikzset{VertexStyle/.append style={scale=1.67}}
  \Vertex[x=0.0,y=0.0,L={A}]{A}
  \Vertex[x=-5.0,y=6.0,L={B}]{B}
  \Vertex[x=5.0,y=6.0,L={C}]{C}
  \Vertex[x=0.0,y=3.0,L={D}]{D}
  \Vertex[x=0.0,y=9.0,L={E}]{E}
  \Vertex[x=0.0,y=12.0,L={F}]{F}
  
  \tikzstyle{LabelStyle}=[scale=1.67,fill=white,sloped]
  \tikzstyle{EdgeStyle}=[post, color=red]
  \Edge[label=$1$](A)(B)
  \Edge[label=$2$](B)(F)
  \Edge[label=$3$](F)(C)
  \Edge[label=$4$](C)(A)

  \tikzstyle{EdgeStyle}=[post, color=blue, dashed, bend left=10]
  \Edge(B)(D)
  \Edge(D)(B)
  \Edge(C)(E)
  \Edge(E)(C)
\end{tikzpicture}
\end{minipage}
\captionof{figure}{Illustration of the argument of Proposition~\ref{saving_prop1}.}
\label{fig:prop1}
\end{center}
\end{figure}
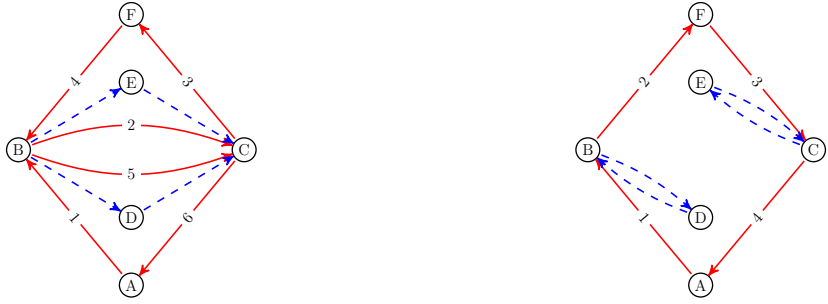

The idea of the transformation $\pi\longmapsto\pi'$ is depicted in Figure~\ref{fig:prop1}. The three conditions on Proposition~\ref{saving_prop1} are met by most of the problem settings in the related literature; this fact justifies their implicit assumption that arc-retraversing solutions do not lead to a strictly lower completion time. Moreover, these conditions are minimal. Indeed, Proposition~\ref{saving_prop1} does not hold when two drones are available (as implied by the solutions shown in Figures \ref{fig:twice_sol} and \ref{fig:once_sol}), or sorties are non-invertible (in particular, the inverted path $\pi_{ji}^{-1}$ might require the drone to fly along infeasible sorties), or sorties serve two customers (see Figure~\ref{fig:thrice_sol}). For example, in the problem setting of \cite{agatz18}, Proposition~\ref{saving_prop1} holds, and therefore it is correct to solve the problem via an IP model that does not allow arc-retraversing solutions. However, if we wanted to generalize the setting as to incorporate multiple drones, or payloads in the energy consumption of a sortie, then excluding arc-retraversing solutions may lead to a strictly higher completion time.
An analogous result to Proposition~\ref{saving_prop1} also holds in a slightly different setting, when the drone does not travel faster than the truck.

\begin{prop}
\label{saving_prop2}
At least one optimal solution is not arc-retraversing if only a single drone is available, all sorties are invertible, $N^{\mathit{tr}}=N$ and $\alpha=1$.
\end{prop}
\begin{proof}
Consider an optimal solution with the properties described in Lemma~\ref{lemma2}, and suppose that it is arc-retraversing. Analogously to Proposition~\ref{saving_prop1}, we denote one of the retraversed arcs by $\left( i,j\right)$, and the truck route by $\pi$. Let $\pi_{ji}$, $\pi_{0i}$ and $\pi_{j0}$ be the analogous paths to those defined in the proof of Proposition~\ref{saving_prop1}; accordingly, $\pi = \pi_{0i}\circ \left( i,j\right) \circ \pi_{ji} \circ \left( i,j\right) \circ \pi_{j0}$. The proof is complete if we show that there is always another solution $\pi'$ with the same completion time, such that the total number of arcs (with multiplicity) that are traversed by the truck and the drone decreases by at least one unit.

Without loss of generality, we can assume that the drone is airborne during at least one truck traversal of $\left( i,j\right)$, either immediately before the traversal of $\pi_{ji}$ or immediately after. Indeed, if not, then we remove arc $\left( i,j\right)$ twice from the route and invert $\pi_{ji}$ (the relevant sorties are invertible by hypothesis, and their support is contained in $\pi_{ji}$ by Lemma~\ref{lemma2}): the solution induced by the truck route $\pi'=\pi_{0i}\circ \pi_{ji}^{-1} \circ \pi_{j0}$ is not longer than $\pi$. Suppose that the drone flies along a path $\pi^{\mathit{dr}}_{ij}$ during the truck traversal of arc $\left( i,j\right)$ immediately before the traversal of $\pi_{ji}$ (the proof is analogous with the traversal of $\left( i,j\right)$ immediately after $\pi_{ji}$). Then, instead of launching the drone along $\pi^{\mathit{dr}}_{ij}$, we instead let the truck traverse it. This is feasible because $N^{\mathit{tr}}=N$, and it does not take (strictly) longer than the drone's sortie, because $\alpha =1$.

After routing the truck along $\pi^{\mathit{dr}}_{ij}$, we follow the original route $\pi$ along $\pi_{ji}\circ  \left( i,j\right)$. The new truck route given by $\pi'=\pi_{0i}\circ \pi^{\mathit{dr}}_{ij} \circ \pi_{ji} \circ \left( i,j\right) \circ \pi_{j0}$ is feasible and leads to a completion time that is not (strictly) larger than that of $\pi$.
Additionally, in this solution, the total number of arcs (with multiplicity) that are traversed by the truck and the drone is strictly smaller than in the original solution.\end{proof}

The idea of the proof of Proposition~\ref{saving_prop2} is shown in Figure~\ref{fig:prop2}. The set of conditions required in Proposition~\ref{saving_prop2} is minimal. Indeed, we would not be able to replace the truck route $\left( i,j\right)$ with $\pi^{\mathit{dr}}_{ij}$ if $N^{\mathit{tr}}\neq N$ or $\alpha >1$. Sorties must be invertible to include cases where an arc $\left( i,j\right)$ is traversed twice by the truck while the drone is not airborne. The route transformation $\pi\longmapsto\pi'$ of Proposition~\ref{saving_prop2} is functional to the \emph{a posteriori} upper bounds described in Section~\ref{subsec_aposteriori}.

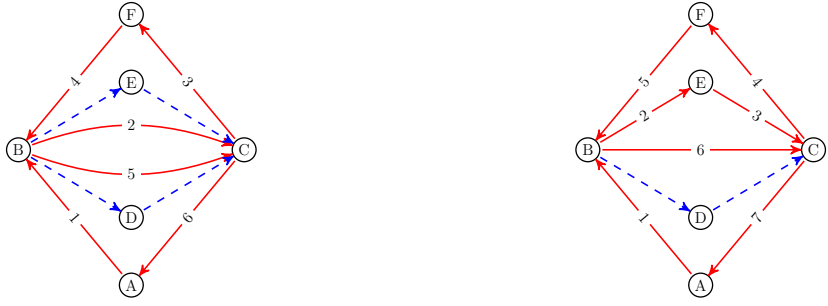
\begin{figure}
\begin{center}
\begin{minipage}{0.45\textwidth}
\centering
\begin{tikzpicture}[scale=0.3,transform shape]
  \tikzset{VertexStyle/.append style={scale=1.67}}
  \Vertex[x=0.0,y=0.0,L={A}]{A}
  \Vertex[x=-5.0,y=6.0,L={B}]{B}
  \Vertex[x=5.0,y=6.0,L={C}]{C}
  \Vertex[x=0.0,y=3.0,L={D}]{D}
  \Vertex[x=0.0,y=9.0,L={E}]{E}
  \Vertex[x=0.0,y=12.0,L={F}]{F}
  
  \tikzstyle{LabelStyle}=[scale=1.67,fill=white,sloped]
  \tikzstyle{EdgeStyle}=[post, color=red]
  \Edge[label=$1$](A)(B)
  \Edge[label=$3$](C)(F)
  \Edge[label=$4$](F)(B)
  \Edge[label=$6$](C)(A)
  
  \tikzstyle{EdgeStyle}=[post, color=red, bend left=20]
  \Edge[label=$2$](B)(C)
  \tikzstyle{EdgeStyle}=[post, color=red, bend right=20]
  \Edge[label=$5$](B)(C)
  
  \tikzstyle{EdgeStyle}=[post, color=blue, dashed]
  \Edge(B)(D)
  \Edge(D)(C)
  \Edge(B)(E)
  \Edge(E)(C)
\end{tikzpicture}
\end{minipage}
\begin{minipage}{0.45\textwidth}
\centering
\begin{tikzpicture}[scale=0.3,transform shape]
  \tikzset{VertexStyle/.append style={scale=1.67}}
  \Vertex[x=0.0,y=0.0,L={A}]{A}
  \Vertex[x=-5.0,y=6.0,L={B}]{B}
  \Vertex[x=5.0,y=6.0,L={C}]{C}
  \Vertex[x=0.0,y=3.0,L={D}]{D}
  \Vertex[x=0.0,y=9.0,L={E}]{E}
  \Vertex[x=0.0,y=12.0,L={F}]{F}
  
  \tikzstyle{LabelStyle}=[scale=1.67,fill=white,sloped]
  \tikzstyle{EdgeStyle}=[post, color=red]
  \Edge[label=$1$](A)(B)
  \Edge[label=$2$](B)(E)
  \Edge[label=$3$](E)(C)
  \Edge[label=$4$](C)(F)
  \Edge[label=$5$](F)(B)
  \Edge[label=$6$](B)(C)
  \Edge[label=$7$](C)(A)

  \tikzstyle{EdgeStyle}=[post, color=blue, dashed]
  \Edge(B)(D)
  \Edge(D)(C)
\end{tikzpicture}
\end{minipage}
\captionof{figure}{Illustration of the argument of Proposition~\ref{saving_prop2}.}
\label{fig:prop2}
\end{center}
\end{figure}

\section{Excluding arc-retraversing solutions}
\label{sec_notincluding}
By excluding arc-retraversing solutions in the TSP-mD, the optimal value might increase. In Sections~\ref{subsec_apriori} and \ref{subsec_aposteriori}, we identify conditions under which \emph{a priori} (i.e., solution-independent) and \emph{a~posteriori} (i.e., solution-dependent) upper bounds hold on such increase, respectively.
\subsection{\emph{A priori} upper bounds on the increase of the completion time}
\label{subsec_apriori}
In Section~\ref{sec_prob_descr}, we have defined two restrictions of the TSP-mD, namely the m-CYCLE and the \mbox{m-CIRCUIT}; by inequality~\eqref{relax}, any upper bound on $\frac{\text{m-CYCLE}}{\text{TSP-mD}}$ also holds on $\frac{\text{m-CIRCUIT}}{\text{TSP-mD}}$.\medskip
$\;$We identify \emph{a~priori} upper bounds on the former ratio $\frac{\text{m-CYCLE}}{\text{TSP-mD}}$ that depend on $m$ and $\alpha$.

\begin{prop}
If the truck can visit every node, then it holds that
\begin{flalign}
\label{already}
\text{m-CYCLE} \leq \left( 1+\alpha m\right) \text{TSP-mD}.&&
\end{flalign}
\end{prop}

Inequality~\eqref{already} follows from $\text{m-CYCLE}\leq \text{TSP}$ and from $\text{TSP}\leq\left( 1+\alpha m\right)\text{TSP-mD}$; the latter inequality was shown by~\cite{wang17} in an analogous setting. We now describe other \emph{a priori} upper bounds, that only depend on $m$ and that dominate inequality~\eqref{already} for all $\alpha>2$. For sake of clarity, we first prove a preliminary result by the following lemma.

\begin{lem}
\label{lemmaapriori}
Let $\pi=\left( i_{1}, \ldots, i_{r}\right)$ be a feasible non-loop sortie. For every $t\in\left\lbrace 2, \ldots, r-1\right\rbrace$, at least one sortie out of \mbox{$\pi_{1}=\left( i_{1}, \ldots, i_{t}, i_{1}\right)$} and $\pi_{2}=\left( i_{r}, \dots, i_{t}, i_{r}\right)$ is feasible.
\end{lem}
\begin{proof} For the sake of notation, we define the quantity $\bar{\ell}\left( p_{1},p_{2}\right) =\sum_{q=p_{1}}^{p_{2}-1}\ell'_{i_{q}i_{q+1}}$, for every indices $p_{1}$ and $p_{2}$ that satisfy $p_{1}< p_{2}\leq r$. Because $\pi$ is feasible, it holds that
\begin{equation}
\label{battery}
    w^{\mathit{dr}}\cdot\bar{\ell}\left( 1, r\right) +\sum_{p=2}^{r-1}\left(w_{i_{p}}\cdot\bar{\ell}\left( 1, p\right)\right) \leq B,
\end{equation}
where the left-hand side quantifies the energy consumption of $\pi$. Let $t$ and $\pi_{1}$ be an index and a sortie like in the hypothesis, respectively. Suppose that $\pi_{1}$ is not feasible. Because $\ell'$ is a metric, it holds that $\bar{\ell}\left( 1, t\right)\geq \ell'_{1t}$. Then, the following quantity is not lower than the energy consumption of $\pi_{1}$:
\begin{equation}
\label{tooheavy}
    2w^{\mathit{dr}}\cdot\bar{\ell}\left( 1, t\right) +\sum_{p=2}^{t}\left( w_{i_{p}}\cdot\bar{\ell}\left( 1, p\right)\right) > B.
\end{equation}
We want to show that $\pi_{2}=\left( i_{r}, \dots, i_{t}, i_{r}\right)$ is feasible. By subtracting the left-hand side of inequality~\eqref{battery} from that of inequality~\eqref{tooheavy}, we deduce that
\begin{equation}
\label{difference}
    w^{\mathit{dr}}\cdot\bar{\ell}\left( 1, t\right) > w^{\mathit{dr}}\cdot\bar{\ell}\left( t, r\right) +\sum_{p=t+1}^{r-1}\left( w_{i_{p}}\cdot\bar{\ell}\left( 1, p\right)\right).
\end{equation}
By dropping the last term on the right-hand side of inequality~\eqref{difference}, we obtain that, for every $p\geq t$,
\begin{equation}
  \label{1p}  \bar{\ell}\left( 1, p\right) \;\geq\; \bar{\ell}\left( 1, t\right) \; >\; \bar{\ell}\left( t, r\right) \;\geq \;\bar{\ell}\left( p, r\right).
\end{equation}
The following chain of inequalities shows that the energy consumption of $\pi_{2}$ is not greater than that of $\pi$, which in turn implies that $\pi_{2}$ is feasible.
\begin{equation}
\label{feasible}
\begin{split}
   &w^{\mathit{dr}}\cdot\left( \ell'_{tr} + \bar{\ell}\left( t, r\right)\right) +\sum_{p=t}^{r-1}\left( w_{i_{p}}\cdot\bar{\ell}\left( p, r\right)\right)\;\leq\\ &\leq\; 2w^{\mathit{dr}}\cdot\bar{\ell}\left( t, r\right) +\sum_{p=t}^{r-1}\left( w_{i_{p}}\cdot\bar{\ell}\left( p, r\right)\right)\;\leq\; 2w^{\mathit{dr}}\cdot\bar{\ell}\left( t, r\right) +\sum_{p=t}^{r-1}\left(w_{i_{p}}\cdot\bar{\ell}\left( 1, p\right)\right)\;\leq\\
    &\leq\; w^{\mathit{dr}}\cdot\bar{\ell}\left( t, r\right) +w^{\mathit{dr}}\cdot\bar{\ell}\left( 1, t\right)+\sum_{p=t}^{r-1}\left(w_{i_{p}}\cdot\bar{\ell}\left( 1, p\right)\right)\;=\; w^{\mathit{dr}}\cdot\bar{\ell}\left( 1, r\right) +\sum_{p=t}^{r-1}\left(w_{i_{p}}\cdot\bar{\ell}\left( 1, p\right)\right)\;\leq\\
    &\leq\; w^{\mathit{dr}}\cdot\bar{\ell}\left( 1, r\right) +\sum_{p=2}^{r-1}\left(w_{i_{p}}\cdot\bar{\ell}\left( 1, p\right)\right).
\end{split}
\end{equation}
The first step of the chain follows from the triangle inequality, while the second and the third steps from inequalities~\eqref{1p}.
\end{proof}

Notice that, with the notation of Lemma~\ref{lemmaapriori}, the sorties $\pi$, $\pi_{1}$, and $\pi_{2}$ satisfy the inequality $\ell'_{\pi_{1}}+\ell'_{\pi_{2}}\leq 2\ell'_{\pi}$ by construction. This fact is crucial to the proof of the following result.

\begin{prop}
It holds that
\begin{flalign}
\label{2m}
\text{m-CYCLE} \leq \left( 1+2m\right) \text{TSP-mD}.&&
\end{flalign}
If the feasible sorties are either loops or contain only one customer, then it further holds that
\begin{flalign}
\label{1m}
\text{m-CYCLE} \leq \left( 1+m\right) \text{TSP-mD}.&&
\end{flalign}
\end{prop}

\begin{proof}
 We first prove inequalities \eqref{2m} and \eqref{1m} for the case $m=1$, and we then adapt the argument to any $m\geq 2$. Suppose that $m=1$: in this case, the proof of inequality~\eqref{2m} is complete if, for every feasible solution, we can replace every non-loop sortie $\pi$ by either a single sortie $\pi_{0}$, or by two sorties $\pi_{1}$ and $\pi_{2}$, that satisfy the following conditions:
\begin{enumerate}
\item[i.] $\pi_{0}$, or $\pi_{1}$ and $\pi_{2}$, are feasible loop sorties;
\item[ii.] $N\!\left[ \pi\right] = N\!\left[ \pi_{0}\right]$, or $N\!\left[ \pi\right] = N\!\left[ \pi_{1}\right]\cup N\!\left[ \pi_{2}\right]$;
\item[iii.]  $\ell'_{\pi_{0}}\leq 2\ell'_{\pi}$, or $\ell'_{\pi_{1}}+\ell'_{\pi_{2}}\leq 2\ell'_{\pi}$, respectively.
\end{enumerate}
Indeed, if we can always replace every non-loop sortie $\pi$ by either the corresponding $\pi_{0}$, or by $\pi_{1}$ and $\pi_{2}$, then we let the truck wait at the starting node of every sortie, and we shortcut its route to eliminate every node revisit. The resulting solution is feasible to the \mbox{1-CYCLE} restriction, and its objective value satisfies $\text{1-CYCLE} \leq 3\cdot\text{TSP-1D} = \left( 1+2m\right)\text{TSP-1D}$, because both the travel times of the truck and the drone are not longer than the \mbox{TSP-1D} objective value, and because of condition~(iii.).

Let then $\pi=\left( i_{1}, \ldots, i_{r}\right)$ be a feasible non-loop sortie. We choose the relevant sorties $\pi_{0}$, or $\pi_{1}$ and $\pi_{2}$, as follows. If at least one sortie out of $\left( i_{1}, \ldots, i_{r-1}, i_{1}\right)$ and $\left( i_{r}, \ldots, i_{2}, i_{r}\right)$ is feasible, we choose one and denote it as $\pi_{0}$. Otherwise, by Lemma~\ref{lemmaapriori}, sorties $\left( i_{1}, i_{2}, i_{1}\right)$ and $\left( i_{r}, i_{r-1}, i_{r}\right)$ are both feasible. Therefore, there exists a unique $t\in\left\lbrace 2, \ldots, r-2\right\rbrace$ such that the sortie $\left( i_{1}, \ldots, i_{t}, i_{1}\right)$ is feasible but $\left( i_{1}, \ldots, i_{t+1}, i_{1}\right)$ is not. We can then choose $\pi_{1}=\left( i_{1}, \ldots, i_{t}, i_{1}\right)$ and, by Lemma~\ref{lemmaapriori}, $\pi_{2}=\left( i_{r}, \ldots, i_{t+1}, i_{r}\right)$. The sorties $\pi_{0}$, or $\pi_{1}$ and $\pi_{2}$, satisfy the aforementioned conditions~(i.)-(iii.).

If a non-loop sortie contains only a single customer, i.e., it is of the form $\left( i_{1},i_{2},i_{3}\right)$, then we can choose the shortest arc between $\left( i_{1},i_{2}\right)$ and $\left( i_{2},i_{3}\right)$, and set $\pi_{0}=\left( i_{1},i_{2},i_{1}\right)$ or $\pi_{0}=\left( i_{3}, i_{2}, i_{3}\right)$ accordingly. Notice that in this case, it holds that $\ell'_{\pi_{0}}\leq \ell'_{\pi}$. If the feasible sorties are either loops or satisfy $r=3$, an argument analogous to that for the $\left( 1+2m\right)$ bound leads to the inequality $\text{1-CYCLE} \leq 2\cdot\text{TSP-1D} = \left( 1+m\right)\text{TSP-1D}$.

If $m\geq 2$, we can repeat the same construction for the sorties of every drone, but this time the truck has to wait up to $m$ times longer for all the drones to perform their sorties.\end{proof}

We complement the results of the previous proposition by showing that inequalities \eqref{2m} and \eqref{1m} are asymptotically tight with a single drone, and that the ratio $\frac{\text{m-CYCLE}}{\text{TSP-mD}}$ cannot be upper-bounded by any constant. In particular, we cannot replace $m$ by any constant in the right-hand sides of inequalities~\eqref{2m} and \eqref{1m}.

\begin{prop}
\label{prop_dependm}
For $m=1$, the bounds \eqref{2m} and \eqref{1m} are asymptotically tight. Moreover, for every~$m$, there exists an instance for which $\frac{\text{m-CYCLE}}{\text{TSP-mD}}\geq m$.
\end{prop}

\begin{proof} We first show that bound \eqref{2m} is asymptotically tight for $m=1$. Consider the set of instances shown in Figure~\ref{fig:inst}: the truck must visit nodes $0$, $1$, $v_{1}$, $v_{2}$, \dots, $v_{k}$, $w_{1}$, $w_{2}$, \dots, $w_{k}$, while the drone must serve nodes $i_{1}$, $i_{2}$, \dots, $i_{2k}$, $j_{1}$, $j_{2}$, \dots, $j_{2k}$, for any $k\geq 1$; the node $0$ is the depot. The lengths of the arcs are either those shown in the figure or the shortest path between the relevant endpoints. The drone travels equally fast as the truck, i.e., $\ell_{ij}=\ell'_{ij}$ for every arc $\left( i,j\right)\in A$. We consider all-zero payloads, and set $L=2+\epsilon$. An optimal \mbox{TSP-1D} solution can be described as follows: the drone flies along sorties of the form $\left(0, i_{q}, j_{q}, 1\right)$ and $\left(1, j_{q}, i_{q}, 0\right)$ for $q\leq 2k$, while the truck traverses the edge $\left\lbrace 0,1\right\rbrace$ for a total of $2k$ times, and sub-routes of the form $\left( 0, v_{q}, 0\right)$ and $\left( 0, w_{q}, 0\right)$ for $q\leq k$, every time the drone is airborne. This solution leads to an objective value of $2k\left( 2+\epsilon\right)$. An optimal solution to the 1-CYCLE can only serve two customers in a single sortie only twice, and the truck can visit one node $v_{q}$ and $w_{q}$ while the drone is airborne only once; hence, the optimal value for the 1-CYCLE amounts to $2\left( 2+\epsilon\right)+4\left( k-1\right)+4\left( 2k-2\right)$. By choosing $\epsilon=\frac{1}{k}$, we get that $\lim_{k\rightarrow\infty}\frac{\text{1-CYCLE}}{\text{TSP-1D}}=3$.

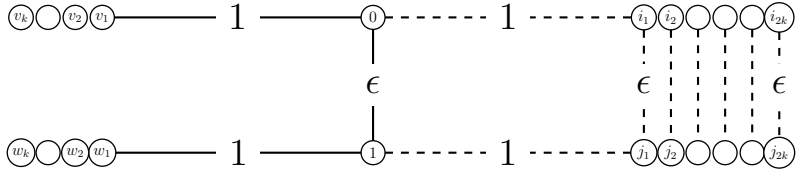
\begin{figure}
    \centering
\begin{tikzpicture}[scale=0.18,transform shape]
  \tikzset{VertexStyle/.append style={scale=2.8}}
  \Vertex[x=0.0,y=0.0,L={0}]{A}
  \Vertex[x=0.0,y=-10.0,L={1}]{B}
  \Vertex[x=20.0,y=0.0,L={$i_{1}$}]{C}
  \Vertex[x=22.0,y=0.0,L={$i_{2}$}]{D}
  \Vertex[x=24.0,y=0.0,L={$\textcolor{white}{.}$}]{E}
  \Vertex[x=26.0,y=0.0,L={$\textcolor{white}{.}$}]{F}
  \Vertex[x=28.0,y=0.0,L={$\textcolor{white}{.}$}]{G}
  \Vertex[x=30.0,y=0.0,L={$i_{2k}$}]{H}
  \Vertex[x=20.0,y=-10.0,L={$j_{1}$}]{I}
  \Vertex[x=22.0,y=-10.0,L={$j_{2}$}]{J}
  \Vertex[x=24.0,y=-10.0,L={$\textcolor{white}{.}$}]{K}
  \Vertex[x=26.0,y=-10.0,L={$\textcolor{white}{.}$}]{L}
  \Vertex[x=28.0,y=-10.0,L={$\textcolor{white}{.}$}]{M}
  \Vertex[x=30.0,y=-10.0,L={$j_{2k}$}]{N}
  \Vertex[x=-20.0,y=0.0,L={$v_{1}$}]{O}
  \Vertex[x=-22.0,y=0.0,L={$v_{2}$}]{P}
  \Vertex[x=-24.0,y=0.0,L={$\textcolor{white}{.}$}]{Q}
  \Vertex[x=-26.0,y=0.0,L={$v_{k}$}]{T}
  \Vertex[x=-20.0,y=-10.0,L={$w_{1}$}]{U}
  \Vertex[x=-22.0,y=-10.0,L={$w_{2}$}]{V}
  \Vertex[x=-24.0,y=-10.0,L={$\textcolor{white}{.}$}]{S}
  \Vertex[x=-26.0,y=-10.0,L={$w_{k}$}]{W}
  
  \tikzstyle{LabelStyle}=[scale=0.67,fill=white]
  \tikzstyle{EdgeStyle}=[scale=10]
  \Edge[label=$\epsilon$](A)(B)
  \tikzstyle{EdgeStyle}=[scale=10, dashed]
  \Edge[label=$1$](A)(C)
  \tikzstyle{EdgeStyle}=[scale=10]
  \Edge[label=$1$](A)(O)
  \tikzstyle{EdgeStyle}=[scale=10, dashed]
  \Edge[label=$1$](B)(I)
  \tikzstyle{EdgeStyle}=[scale=10]
  \Edge[label=$1$](B)(U)
  \tikzstyle{EdgeStyle}=[scale=10, dashed]
  \Edge[label=$\epsilon$](C)(I)
  \Edge(D)(J)
  \Edge(E)(K)
  \Edge(F)(L)
  \Edge(G)(M)
  \Edge[label=$\epsilon$](H)(N)
\end{tikzpicture}
\caption{Illustration of the set of instances in the proof of Proposition~\ref{prop_dependm}, parametrized by $k\geq 1$.}
    \label{fig:inst}
\end{figure}

Analogously, bound \eqref{1m} is asymptotically tight for $m=1$. Consider the instance shown in Figure~\ref{fig:inst} and remove nodes $i_{1}$, \dots, $i_{2k}$, $v_{1}$, \dots, $v_{k}$, $j_{k+1}$, $j_{k+2}$, \dots, $j_{2k}$. Again, the drone travels as fast as the truck, we consider all-zero payloads, and set $L=2$. The only allowed sorties are then loops of the form $\left( 1, j_{q}, 1\right)$, for $q\leq k$. An optimal TSP-1D solution can visit nodes $w_{q}$ for $q\leq k$ while the drone is airborne, and this leads to an optimal value equal to $2\epsilon + 2k$. At the same time, the optimal value for the 1-CYCLE amounts to $2\epsilon + 4k$, which leads to an asymptotic bound of $2$.

Finally, we need to show that the ratio $\frac{\text{m-CYCLE}}{\text{TSP-mD}}$ cannot be upper-bounded by any constant. Let $m>1$, and consider $m$ nodes at the vertices of a regular $m$-tope with edge length $\epsilon$; for each of these nodes $i_{1}, \dots, i_{m}$, we consider nodes $j_{1}, \dots, j_{m}$ such that the length $\ell_{i_{q}j_{q}}=\frac{1}{2}$ for all $q\leq m$. The length of all the other arcs is set to the shortest path between their endpoints. We choose $i_{1}$ as the depot, and set proportional metrics with $\alpha=2$ and $L=1$. The only sorties that can be selected in an optimal solution are loops of the form $\left( i_{q}, j_{q}, i_{q}\right)$ for $q\leq m$. The optimal m-CYCLE objective amounts to $m\left( 1+\epsilon\right)$, while the optimal TSP-mD amounts to $1+m\epsilon$.\end{proof}

\subsection{\emph{A posteriori} upper bounds on the increase of the completion time}
\label{subsec_aposteriori}

In this section, we describe a method to provide a solution-dependent upper bound on the increase of the completion time due to the exclusion of arc-retraversing solutions from the solution space of the TSP-mD. Consider an optimal solution $\mathcal{S}$ to the TSP-mD and let $h=\left( h_{1}, h_{2}, \ldots, h_{\left| A\right|}\right)$ be a vector in $\mathbb{Z}_{+}^{\left| A\right|}$, whose components $h_{a}$ represent the number of times the truck traverses arc $a$ in $\mathcal{S}$, for any $a\in A$. We call $h$ the \emph{retraversing vector} of $\mathcal{S}$.

\begin{prop}
\label{saving_prop4}
Suppose that $N^{\mathit{tr}}=N$, $\alpha\geq 1$ and only a single drone is available. Then, for any optimal solution whose retraversing vector is $h\in \mathbb{Z}_{+}^{\left| A\right|}$, it holds that
\begin{equation}
\text{TSP-mD}\; +\; \left(\sum_{a\in A}\left( h_{a}-1\right)^{+}\right)\cdot\left(\alpha -1\right)\cdot L\; \geq \text{m-CIRCUIT}.
\end{equation}
\end{prop}
\begin{proof}
The proof follows from transforming the truck route analogously to the proof of Proposition~\ref{saving_prop2} for every arc $a$ such that $h_{a}>1$. Any such route transformation removes one arc-retraversal and instead routes the truck along a sortie of the initial solution; the corresponding increases in the objective value are at most $\left(\alpha -1\right)\cdot L$, because no sortie can be strictly longer than $L$.
\end{proof}

When the conditions of Proposition~\ref{saving_prop4} are met, and given an optimal solution to the TSP-mD with a retraversing vector $h$, we can conclude that the percentage increase in the completion time due to not allowing arc-retraversing solutions satisfies the following inequality:
\begin{equation}
\label{percentage_decrease}
\frac{\text{m-CIRCUIT}-\text{TSP-mD}}{\text{TSP-mD}}\; \leq\;\frac{\left(\sum_{a\in A}\left( h_{a}-1\right)^{+}\right)\cdot\left(\alpha -1\right)\cdot L}{\text{LB}},
\end{equation}
where LB is a lower bound for the TSP-mD, that can be easily pre-computed as follows:
\begin{prop}
\label{prop_LB}
With a single drone and $N^{\mathit{tr}}=N$, let $f\in N$ be the furthest node from the depot. Then, the quantity
\begin{equation}
\label{LB}
\mathrm{LB}\; =\;\mathrm{min}\left\lbrace \ell_{0i} + \ell_{0j}+\mathrm{max}\left\lbrace\ell_{ij},\, \ell'_{if}+\ell'_{fj}\right\rbrace\; :\; i,j\in N\text{ and }\left( i,f,j\right)\in P\right\rbrace
\end{equation}
is a lower bound on the objective value for the TSP-mD.
\end{prop}
\begin{proof}
The node $f$ must be visited by either the truck or the drone; hence, the completion time must be greater or equal than the minimum time it takes to visit $f$. If $f$ is visited by the truck in an optimal solution, then such a time is $2\ell_{0f}$, which is included in \eqref{LB} by choosing $i=j=f$. If $f$ is instead visited by the drone, without loss of generality, we only have to consider the feasible sorties of the form $\left( i,f,j\right)$ (with possibly $i=j=0$ that leads to the drone visiting $f$ directly from the depot). Indeed, sorties serving strictly more nodes cannot lead to a shorter time to visit $f$. Once a sortie of the form $\left( i,f,j\right)$ is used, the truck needs to traverse the arcs $\left( 0,i\right)$, $\left( i,j\right)$ and $\left( j,0\right)$; the drone is airborne while the truck travels along $\left( i,j\right)$, hence the quantity $\mathrm{max}\left\lbrace\ell_{ij},\, \ell'_{if}+\ell'_{fj}\right\rbrace$ in \eqref{LB}.

\end{proof}
Propositions \ref{saving_prop4} and \ref{prop_LB} provide an upper bound on the percentage increase of the completion time due to not allowing arc-retraversing solutions, with a single drone and $N^{\mathit{tr}}=N$. Consider, for example, the instance corresponding to Figures~\ref{fig:thrice_onlyonce_sol} and \ref{fig:thrice_sol}, which satisfies $N^{\mathit{tr}}\neq N$ and $\alpha=1$. In Section~\ref{time-saving_subsection}, we have shown that, in this instance, the percentage increase of the completion time is at least $\frac{1084.09-1050.36}{1050.36}=3.21\%$. If we instead set $N^{\mathit{tr}}=N$ and $\alpha =\frac{4}{3}$, like for Figures~\ref{fig:twice_sol} and \ref{fig:once_sol}, we can show that, by inequality~\eqref{percentage_decrease}, the percentage increase is at most $\frac{3\cdot\nicefrac{1}{3}\cdot 43.04}{2\cdot 228.09+410}=4.97\%$. Indeed, because of the metric $\ell'$ and the parameter $L$, the truck must visit nodes F and I in any feasible solution. Consequently, the minimum amount of time it takes to visit F and I, namely, $2\cdot 228.09+410$, is a lower bound on the objective value of the TSP-mD.

\section{Approximability results for the TSP-mD}
\label{sec_approx}
Contrary to the well-known results for the metric TSP, we prove that the TSP-mD is not approximable within any constant factor. For the special case where the truck is allowed to visit all the nodes, we identify a (non-constant) approximation factor explicitly, via a result that relies on the Christofides heuristic for the metric TSP.

\subsection{Inapproximability within any constant factor}
\label{subsec_inapprox}
In this section, we restrict ourselves to the special case where the truck and the drone metrics are proportional, i.e., $\ell_{a}=\alpha\cdot\ell'_{a}$ for any arc $a\in A$, and $\alpha\geq 1$. In particular, it suffices to specify the metric~$\ell$ and the parameter $\alpha$ to also determine $\ell'$.
In this problem setting, we prove that the TSP-mD is not approximable within any constant factor, even when either $\alpha$ or $m$ is given. The result is based on a reduction to the TSP-mD of the Minimum Set Cover Problem, which is briefly defined below.
\medskip

\noindent\textbf{Definition.} Let $X$ be a finite set and $\mathcal{S}$ be a finite collection of subsets of $X$ such that \mbox{$\cup_{U\in\mathcal{S}}U = X$}. The Minimum Set Cover Problem (MSC) is the problem of finding \begin{equation}\text{min}\!\left\lbrace \left|\mathcal{S}'\right|\;\; :\;\; \mathcal{S}'\subseteq\mathcal{S}\text{ and }\cup_{U\in\mathcal{S}'}U = X\right\rbrace.\end{equation}
\medskip

\noindent Unless P=NP, the MSC is well known not to be approximable within a factor $c\cdot\ln\!\left(\left| X\right|\right)$ via any polynomial-time heuristic algorithm, for any $c\in\left( 0,1\right)$. In particular, the MSC is not approximable within any constant factor \citep{Dinur14}.

\begin{prop}
\label{prop_inapprox}
Unless P=NP, no polynomial-time heuristic algorithm can approximate the \mbox{TSP-mD} with either $\alpha$ or $m$ given within a constant factor.
\end{prop}
\begin{proof}
The proof consists of five steps, which we enumerate as follows:

\noindent \textbf{i.} for every MSC instance $\left( X, \mathcal{S}\right)$, $\alpha\geq 1$ and $m$, we can construct an instance $I^{X,\mathcal{S}}$ for the TSP-mD in polynomial time;

\noindent \textbf{ii.} for every $m$ and every $\alpha\geq 1$, if there exists a polynomial-time heuristic for the TSP-mD, then there exists a polynomial-time heuristic for the MSC such that for any MSC instance $\left( X, \mathcal{S}\right)$ it holds that
    \begin{equation}
    \text{HEUR}_{\text{MSC}}\!\left( X, \mathcal{S}\right)\leq\text{HEUR}_{\text{TSP-mD}}\!\left( I^{X,\mathcal{S}}\right);
    \end{equation}
    
\noindent \textbf{iii.} for every MSC instance $\left( X, \mathcal{S}\right)$ and for every $\alpha\geq 1$, there exist an $m$ such that
    \begin{equation}
    \text{TSP-mD}\!\left( I^{X,\mathcal{S}}\right)\leq 2\cdot\text{MSC}\!\left( X, \mathcal{S}\right) ;
    \end{equation}
    
\noindent \textbf{iv.} for every MSC instance $\left( X, \mathcal{S}\right)$, every $\epsilon >0$ and every $m\geq 1$, there exist an $\alpha\geq 1$ such that
    \begin{equation}
    \text{TSP-mD}\!\left( I^{X,\mathcal{S}}\right)\leq \left( 1+\epsilon\right)\cdot\text{MSC}\!\left( X, \mathcal{S}\right) ;
    \end{equation}
    
\noindent \textbf{v.} by contradiction: for both the cases where either $\alpha\geq 1$ or $m$ is given, if there existed a polynomial-time heuristic that approximates the TSP-mD within a constant factor, then the same would apply to the MSC, contradicting its inapproximability properties.

\medskip
\noindent\emph{Proof of} \textbf{i.} Let $\left( X, \mathcal{S}\right)$ be an instance for the MSC. We construct an instance $I^{X,\mathcal{S}}$ for the TSP-mD as follows. The set $N$ of nodes consists of the depot $0$, a node $i_{S}$ for every $S\in\mathcal{S}$ and a node $j_{x}$ for every $x\in X$. We set the length of the arc $\ell_{0i_{S}}=\nicefrac{1}{2}$ for every $S\in\mathcal{S}$ and $\ell_{i_{S}j_{x}}=\nicefrac{1}{2}$ for every $S\in\mathcal{S}$ and every $x\in S$. All the other arcs have a length equal to the shortest path between their endpoints. We set $N^{\mathit{dr}}=N^{\mathit{tr}}=N$, $w_{i}=0$ for every $i\in N\smallsetminus\left\lbrace 0\right\rbrace$, and $L=\nicefrac{1}{\alpha}$. Figure~\ref{fig:inapprox} illustrates the construction of $I^{X,\mathcal{S}}$.

\medskip
\noindent\emph{Proof of} \textbf{ii.} Let $\text{HEUR}_{\text{TSP-mD}}$ be a polynomial-time heuristic for the TSP-mD. Given an MSC instance $\left( X, \mathcal{S}\right)$, we construct a feasible MSC solution $\mathcal{S}'$ as follows. We compute $\text{HEUR}_{\text{TSP-mD}}\!\left( I^{X,\mathcal{S}}\right)$. Then, for every $S\in\mathcal{S}$, we include $S$ in $\mathcal{S}'$ if the corresponding node $i_{S}$ is visited by the truck. For all $x\in X$, if $j_{x}$ is visited by the truck, we arbitrarily choose a set $S\ni x$; if $S$ was not yet included in $\mathcal{S}'$, then we include it in $\mathcal{S}'$. By construction, for all $x\in X$ we have included at least one set $S\in\mathcal{S}$ such that $x\in S$. Moreover, because the length of the arcs satisfies the triangle inequality, it holds that $\left|\mathcal{S}'\right|\leq\text{HEUR}_{\text{TSP-mD}}\!\left( I^{X,\mathcal{S}}\right)$. This construction defines a polynomial-time heuristic for the MSC with the required property.

\medskip
\noindent\emph{Proof of} \textbf{iii.} Let $\alpha\geq 1$ be given and let $\left( X, \mathcal{S}\right)$ be an instance for the MSC. We choose $m=\left| X\right| + \left|\mathcal{S}\right|$. Moreover, let $\mathcal{S'}\subseteq\mathcal{S}$ be an optimal sub-collection for the MSC, and let $I^{X,\mathcal{S}}$ be the TSP-mD instance from step (i.). We construct a feasible solution for the TSP-mD as follows. The truck travels $\left|\mathcal{S}'\right|$ times back and forth from the depot to the nodes $i_{S}$, for all $S\in\mathcal{S'}$. Every time the truck reaches a node $i_{S}$, a drone is launched onto a sortie $\left( i_{S},j_{x},i_{S}\right)$ for every $x\in S$, while the truck stops and waits. At the same time, for every set $S\in\mathcal{S}\smallsetminus\mathcal{S}'$, we launch one drone onto sortie $\left( 0,i_{S},0\right)$. This solution is feasible and leads to an objective value of $\left( 1+\frac{1}{\alpha}\right)\left(\text{MSC}\!\left( X, \mathcal{S}\right)\right)\leq 2\cdot\text{MSC}\!\left( X, \mathcal{S}\right)$.

\medskip
\noindent\emph{Proof of} \textbf{iv.} Let $m\geq 1$ and $\epsilon >0$ be given, and let $\left( X, \mathcal{S}\right)$ be an instance for the MSC. We now choose $\alpha=\frac{\left| X\right| + \left|\mathcal{S}\right|}{\epsilon}$. With a construction similar to that of step (iii.), we get a feasible solution for the TSP-mD that uses only one drone and leads to an objective value $\frac{1}{\alpha}\left(\left|\mathcal{S}\right|-\text{MSC}\!\left( X, \mathcal{S}\right)\right) +\text{MSC}\!\left( X, \mathcal{S}\right) +\frac{\left| X\right|}{\alpha}\leq \epsilon+\text{MSC}\!\left( X, \mathcal{S}\right)\leq\left( 1+\epsilon\right)\cdot\text{MSC}\!\left( X, \mathcal{S}\right)$.

\medskip
\noindent\emph{Proof of} \textbf{v.} This step is analogous for both the cases where either $\alpha\geq 1$ or $m$ is given. If a polynomial-time heuristic for the TSP-mD and a constant $k>1$ existed such that $\text{HEUR}_{\text{TSP-mD}}\leq k\text{TSP-mD}$ for all the instances, then the following chain of inequalities would hold by all the previous steps, for any MSC instance $\left( X, \mathcal{S}\right)$:
\begin{equation}
\text{HEUR}_{\text{MSC}}\!\left( X, \mathcal{S}\right)\leq\text{HEUR}_{\text{TSP-mD}}\!\left( I^{X,\mathcal{S}}\right)\leq k\text{TSP-mD}\!\left( I^{X,\mathcal{S}}\right)\leq 2k\text{MSC}\!\left( X, \mathcal{S}\right) .
\end{equation}
\end{proof}

Notice that, because of the construction of the TSP-mD instance described in step (i.) of the proof, Proposition~\ref{prop_inapprox} also holds when $N=N^{\mathit{tr}}=N^{\mathit{dr}}$. The same reduction of the MSC can be used to derive a further double-logarithmic factor inapproximability for the TSP-mD.

\begin{figure}
    \centering
\begin{tikzpicture}[scale=0.18,transform shape]
  \tikzset{VertexStyle/.append style={scale=2.8}}
  \Vertex[x=-20.0,y=0.0,L={0}]{A}
  
  \Vertex[x=0.0,y=15.0,L={$i_{1}$}]{B}
  \Vertex[x=0.0,y=10.0,L={$i_{2}$}]{C}
  \Vertex[x=0.0,y=5.0,L={$\textcolor{white}{.}$}]{D}
  \Vertex[x=0.0,y=0.0,L={$\textcolor{white}{.}$}]{E}
  \Vertex[x=0.0,y=-5.0,L={$\textcolor{white}{.}$}]{F}
  \Vertex[x=0.0,y=-10.0,L={$\textcolor{white}{.}$}]{G}
  \Vertex[x=0.0,y=-15.0,L={$i_{\left|\mathcal{S}\right|}$}]{H}

  \Vertex[x=20.0,y=10.0,L={$j_{1}$}]{I}
  \Vertex[x=20.0,y=5.0,L={$j_{2}$}]{J}
  \Vertex[x=20.0,y=0.0,L={$\textcolor{white}{.}$}]{K}
  \Vertex[x=20.0,y=-5.0,L={$\textcolor{white}{.}$}]{L}
  \Vertex[x=20.0,y=-10.0,L={$j_{\left| X\right|}$}]{M}
  
  \Vertex[x=-20.0,y=20.0,L={$0$}]{N}
  \Vertex[x=0.0,y=20.0,L={$\mathcal{S}$}]{O}
  \Vertex[x=20.0,y=20.0,L={$X$}]{P}
  
  \tikzstyle{LabelStyle}=[scale=0.75,fill=white]
  \tikzstyle{EdgeStyle}=[scale=10, <->]
  \Edge[label=$\frac{1}{2}$](N)(O)
  
  \tikzstyle{EdgeStyle}=[scale=10, <->, dashed]
  \Edge[label=$\frac{1}{2}$](O)(P)
  
  \tikzstyle{EdgeStyle}=[scale=10]
  \Edge(A)(B)
  \Edge(A)(C)
  \Edge(A)(D)
  \Edge(A)(E)
  \Edge(A)(F)
  \Edge(A)(G)
  \Edge(A)(H)
  
  \tikzstyle{EdgeStyle}=[post, scale=10, dashed]
  \Edge(B)(I)
  \Edge(B)(K)
  \Edge(B)(M)
  \Edge(C)(I)
  \Edge(C)(J)
  \Edge(H)(L)
  \Edge(H)(M)
\end{tikzpicture}
\caption{Illustration of the TSP-mD instance $I^{\left( X, \mathcal{S}\right)}$ in step (i.) of the proof of Proposition~\ref{prop_inapprox}.}
    \label{fig:inapprox}
\end{figure}
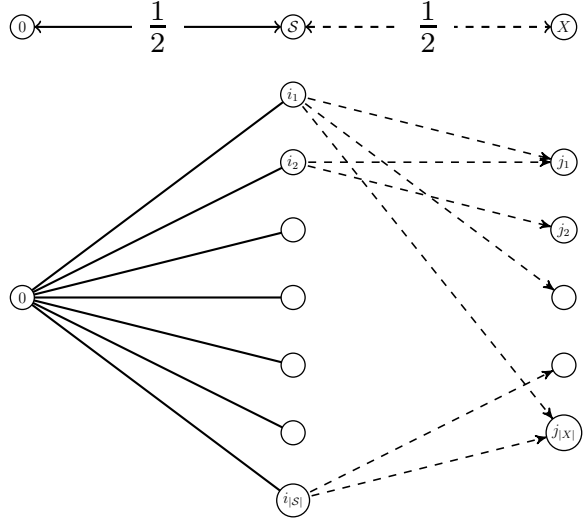

\begin{cor}
Unless P=NP, the TSP-mD is not approximable within a factor of $c\cdot\ln\ln\left(\left| N\right|\right)$, for any instance with $N$ as set of nodes and any $c\in\left( 0,\frac{1}{2}\right)$.
\end{cor}
\begin{proof}For every MSC instance $\left( X,\mathcal{S}\right)$ with $\left| X\right|\geq 2$, it holds that $\left| \mathcal{S}\right|\leq 2^{\left| X\right|}$ and
\begin{equation}
c\cdot\ln\ln\!\left( 1+\left| X\right| + \left|\mathcal{S}\right|\right)\leq c\cdot\ln\ln\left( 2^{2\left| X\right|}\right)\leq 2c\cdot\ln\!\left(\left| X\right|\right) + c\cdot\ln\ln 2\leq 2c\cdot\ln\!\left(\left| X\right|\right).
\end{equation}
Suppose that a polynomial-time heuristic for the TSP-mD existed that approximates it within a factor \mbox{$c\cdot\ln\ln\!\left(\left| N\right|\right)$} for $c\in\left( 0,\frac{1}{2}\right)$. Then, analogously to step (iv.) in the proof of Proposition~\ref{prop_inapprox}, it holds that 
\begin{equation}
\begin{split}
\text{HEUR}_{\text{MSC}}\!\left( X,\mathcal{S}\right)&\leq\text{HEUR}_{\text{TSP-mD}}\!\left( I^{X,\mathcal{S}}\right)\leq c\cdot\ln\ln\!\left( 1+\left| X\right| + \left|\mathcal{S}\right|\right)\text{TSP-mD}\!\left( I^{X,\mathcal{S}}\right)\leq \\&\leq 2c\cdot\ln\!\left(\left| X\right|\right)\text{TSP-mD}\!\left( I^{X,\mathcal{S}}\right)\leq 2c\left( 1+\epsilon\right)\cdot\ln\!\left(\left| X\right|\right)\text{MSC}\!\left(X,\mathcal{S}\right),
\end{split}
\end{equation}
for every $\epsilon >0$, which contradicts the approximability properties of the MSC.\end{proof}

\subsection{An approximation within a non-constant factor}
\label{subsec_approx}
If the truck can visit all the nodes, i.e., $N^{\text{tr}}=N$, then it holds that $\text{TSP}\leq \left( 1+\alpha m\right)\text{TSP-mD}$. The proof can be found in~\cite{wang17}. This bound is tight, i.e., there exist instances for which the bound is satisfied to equality. Notice that when $N^{\text{tr}}=N$, a Hamiltonian cycle induces a feasible solution for the TSP-mD; this fact leads us to establish the following (non-constant) approximation factor:

\begin{prop}
\label{christof}
The TSP-mD is approximable within a $\frac{3}{2}\left( 1+\alpha m\right)$ factor in polynomial time if the truck can visit all the nodes.
\end{prop}

The proof follows from the fact that the Christofides' heuristic is well known to produce a $\frac{3}{2}$~approximation cycle for the TSP. Notice that, by Proposition~\ref{prop_inapprox}, we cannot replace $\alpha$ or $m$ by a constant in this approximation factor.

\section{Conclusion}
\label{sec_conclusion}
In this work, we have introduced arc-retraversing solutions for the TSP with Drones, and showed via two Euclidean instances that ignoring them may increase the completion time. On the other hand, we have identified conditions, which are commonly met in the drone routing literature, under which it suffices to consider solutions that are not arc-retraversing. By excluding arc retraversals when these conditions are not satisfied, the optimal value may increase; we have described cases where \emph{a~priori} and \emph{a~posteriori} upper bounds on such increase hold. Furthermore, we have shown that no polynomial-time heuristic algorithm exists that approximates the metric TSP-mD within a constant factor, unless P=NP. We have explicitly found a non-constant approximation factor for the special case without restrictions on the truck route.

Improvements on our work might be possible by extending the results of Section~\ref{results_subsec} to the multiple drones case; the extension, however, does not appear to immediately follow from any results or ideas of this paper. Potential future work might include the development of a branch-and-price algorithm where the columns correspond to feasible drone operations, and a classification of the instances in the literature where all the optimal solutions are arc-retraversing.

\section*{Acknowledgments}
The authors wish to thank Ben Hermans (KU Leuven) and Ivana Ljubi\'c (ESSEC Business School, Paris) for their valuable suggestions on a draft of this work.

\appendix
\section{Description of the instances}
\label{app_inst}
An instance of the TSP-mD is defined by the distance matrices representing the metrics $\ell$ and $\ell'$, by the drone weight $w^{\mathit{dr}}$ and the payload of the nodes, by the parameters $m$ and $B$, and by the subsets of nodes $N^{\mathit{tr}}$ and $N^{\mathit{dr}}$. If the metrics are proportional, $\ell'$ is completely determined by $\ell$ and the parameter $\alpha$; if the payloads are null, it suffices to specify $L$ instead of the parameters $w^{\mathit{dr}}$ and $B$.
\subsection{Instance of Figure~\ref{fig:feas_sol}}
\label{subapp_feas_sol}
\begin{table}[!ht]
    \centering
    \tiny
    \renewcommand{\arraystretch}{2}
    \begin{tabular}{ccccccccccc}
        $\ell$ & A & B & C & D & E & F & G & H & I & J \\ \hline
        A & 0 & 13.89 & 21.02 & 32.56 & 17.2 & 14.14 & 11.4 & 26.42 & 22.02 & 23.09 \\ \hline
        B & 13.89 & 0 & 12.37 & 19.21 & 31.06 & 22.2 & 16.76 & 22.83 & 11.66 & 24.21 \\ \hline
        C & 21.02 & 12.37 & 0 & 15.3 & 37.01 & 21.02 & 28.07 & 34.93 & 22.2 & 16.28 \\ \hline
        D & 32.56 & 19.21 & 15.3 & 0 & 49.68 & 36.06 & 35.36 & 35.01 & 21.1 & 31 \\ \hline
        E & 17.2 & 31.06 & 37.01 & 49.68 & 0 & 20.4 & 21.02 & 37.12 & 37.64 & 32.76 \\ \hline
        F & 14.14 & 22.2 & 21.02 & 36.06 & 20.4 & 0 & 25.5 & 40.22 & 33.24 & 12.37 \\ \hline
        G & 11.4 & 16.76 & 28.07 & 35.36 & 21.02 & 25.5 & 0 & 16.49 & 18.03 & 34.01 \\ \hline
        H & 26.42 & 22.83 & 34.93 & 35.01 & 37.12 & 40.22 & 16.49 & 0 & 14.04 & 46.1 \\ \hline
        I & 22.02 & 11.66 & 22.2 & 21.1 & 37.64 & 33.24 & 18.03 & 14.04 & 0 & 35.81 \\ \hline
        J & 23.09 & 24.21 & 16.28 & 31 & 32.76 & 12.37 & 34.01 & 46.1 & 35.81 & 0 \\ \hline
    \end{tabular}
    \caption{Matrix representing the metric $\ell$ of the instance shown in Figure~\ref{fig:feas_sol}.}
    \label{tab:feas_sol}
\end{table}
In Table~\ref{tab:feas_sol}, we show the distance matrix corresponding to the metric $\ell$ in the instance of Figure~\ref{fig:feas_sol}. The instance is further defined by $m=3$, proportional metrics with $\alpha=\frac{4}{3}$, $L=40.00$, $N=N^{\mathit{dr}}=N^{\mathit{tr}}$ and all-zero payloads.

\subsection{Instance of Figures~\ref{fig:thrice_onlyonce_sol} and \ref{fig:thrice_sol}}
\label{subapp_thrice_onlyonce_sol}
\begin{table}[!ht]
    \centering
    \tiny
    \setlength{\tabcolsep}{4.5pt}
    \renewcommand{\arraystretch}{2}
    \begin{tabular}{cccccccccccccccc}
        $\ell$ & A & B & C & D & E & F & G & H & I & J & K & L & M & N & O \\ \hline
        A & 0 & 79.15 & 79.15 & 84.18 & 84.18 & 228.09 & 100.13 & 100.13 & 228.09 & 109.01 & 109.01 & 116.13 & 116.13 & 120.85 & 120.85 \\ \hline
        B & 79.15 & 0 & 0.2 & 7.26 & 7.41 & 205.96 & 21.42 & 21.46 & 206.16 & 29.87 & 29.88 & 37.24 & 37.27 & 41.7 & 41.7 \\ \hline
        C & 79.15 & 0.2 & 0 & 7.41 & 7.26 & 206.16 & 21.46 & 21.42 & 205.96 & 29.88 & 29.87 & 37.27 & 37.24 & 41.7 & 41.7 \\ \hline
        D & 84.18 & 7.26 & 7.41 & 0 & 11 & 200.14 & 16.01 & 19.14 & 211.11 & 25.35 & 25.91 & 32 & 33.84 & 37.24 & 37.27 \\ \hline
        E & 84.18 & 7.41 & 7.26 & 11 & 0 & 211.11 & 19.14 & 16.01 & 200.14 & 25.91 & 25.35 & 33.84 & 32 & 37.27 & 37.24 \\ \hline
        F & 228.09 & 205.96 & 206.16 & 200.14 & 211.11 & 0 & 200 & 210 & 410 & 203.9 & 206.5 & 200.14 & 211.11 & 205.96 & 206.16 \\ \hline
        G & 100.13 & 21.42 & 21.46 & 16.01 & 19.14 & 200 & 0 & 10 & 210 & 9.73 & 10.99 & 16.01 & 19.14 & 21.42 & 21.46 \\ \hline
        H & 100.13 & 21.46 & 21.42 & 19.14 & 16.01 & 210 & 10 & 0 & 200 & 10.99 & 9.73 & 19.14 & 16.01 & 21.46 & 21.42 \\ \hline
        I & 228.09 & 206.16 & 205.96 & 211.11 & 200.14 & 410 & 210 & 200 & 0 & 206.5 & 203.9 & 211.11 & 200.14 & 206.16 & 205.96 \\ \hline
        J & 109.01 & 29.87 & 29.88 & 25.35 & 25.91 & 203.9 & 9.73 & 10.99 & 206.5 & 0 & 2.6 & 8.16 & 9.76 & 11.91 & 11.93 \\ \hline
        K & 109.01 & 29.88 & 29.87 & 25.91 & 25.35 & 206.5 & 10.99 & 9.73 & 203.9 & 2.6 & 0 & 9.76 & 8.16 & 11.93 & 11.91 \\ \hline
        L & 116.13 & 37.24 & 37.27 & 32 & 33.84 & 200.14 & 16.01 & 19.14 & 211.11 & 8.16 & 9.76 & 0 & 11 & 7.26 & 7.41 \\ \hline
        M & 116.13 & 37.27 & 37.24 & 33.84 & 32 & 211.11 & 19.14 & 16.01 & 200.14 & 9.76 & 8.16 & 11 & 0 & 7.41 & 7.26 \\ \hline
        N & 120.85 & 41.7 & 41.7 & 37.24 & 37.27 & 205.96 & 21.42 & 21.46 & 206.16 & 11.91 & 11.93 & 7.26 & 7.41 & 0 & 0.2 \\ \hline
        O & 120.85 & 41.7 & 41.7 & 37.27 & 37.24 & 206.16 & 21.46 & 21.42 & 205.96 & 11.93 & 11.91 & 7.41 & 7.26 & 0.2 & 0 \\ \hline
    \end{tabular}
    \caption{Matrix representing the metric $\ell$ of the instance shown in Figures~\ref{fig:thrice_onlyonce_sol} and \ref{fig:thrice_sol}.}
    \label{tab:thrice_onlyonce_sol}
\end{table}
In Table~\ref{tab:thrice_onlyonce_sol}, we show the distance matrix corresponding to the metric $\ell$ of the instance shown in Figures~\ref{fig:thrice_onlyonce_sol} and~\ref{fig:thrice_sol}. The instance is further defined by $m=1$, proportional metrics with $\alpha=1$, $L=43.04$, $N^{\mathit{tr}}=\left\lbrace A,F,G,H, I\right\rbrace$ and $N^{\mathit{dr}}=\left\lbrace B,C,D,E,J,K,L,M,N,O\right\rbrace$, and all-zero payloads.

\subsection{Instance of Figures~\ref{fig:twice_sol} and \ref{fig:once_sol}}
\label{subapp_twice_sol}
\begin{table}[!ht]
    \centering
    \tiny
    \renewcommand{\arraystretch}{2}
    \begin{tabular}{cccccc}
        $\ell$ & A & B & C & D & E \\ \hline
        A & 0 & 10.00 & 10.00 & 24.00 & 24.00\\ \hline
        B & 10.00 & 0 & 14.14 & 14.00 & 26.00\\ \hline
        C & 10.00 & 14.14 & 0 & 26.00 & 14.00\\ \hline
        D & 24.00 & 14.00 & 26.00 & 0 & 33.94\\ \hline
        E & 24.00 & 26.00 & 14.00 & 33.94& 0\\ \hline
    \end{tabular}
    \caption{Matrix representing the metric $\ell$ of the instance shown in Figures~\ref{fig:twice_sol} and \ref{fig:once_sol}.}
    \label{tab:twice_sol}
\end{table}
In Table~\ref{tab:twice_sol}, we show the distance matrix corresponding to the metric $\ell$ of the instance shown in Figures~\ref{fig:twice_sol} and \ref{fig:once_sol}. The instance is further defined by $m=2$, proportional metrics with $\alpha=\frac{4}{3}$, $L=21.00$, $N=N^{\mathit{dr}}=N^{\mathit{tr}}$ and all-zero payloads.\medskip

\section{MILP model}
\label{app_model}
The proofs of Propositions~\ref{saving_prop1} and \ref{saving_prop2} are complete if we provide an MILP model for the TSP-mD: we propose one based on the formulation for the Time-Dependent TSP (TDTSP) by~\cite{Picard78}. In the TDTSP, the cost of using an arc  depends on its position in the route. Similarly, we describe the truck route $\pi$ of a TSP-mD solution as a sequence of (possibly repeated) arcs in $N\times N=A\cup\left\lbrace\left( i,i\right)\; : \; i\in N\right\rbrace$, where the arcs of the form $\left\lbrace\left( i,i\right)\; : \; i\in N\right\rbrace$ represent the truck waiting at a node $i\in N$ for a drone. Furthermore, we describe a sortie as a sequence of nodes and a couple of positive numbers $t_{1}\leq t_{2}$; the nodes of the sortie represent (in the order) the starting, the served and the ending locations, while $t_{1}$ and $t_{2}$ correspond to the positions in the truck route of the arcs that are traversed immediately after launching and right before retrieving the relevant drone, respectively. We refer to $t_{1}$ and $t_{2}$ as starting and ending positions of the sortie, respectively. For example, in the solution shown in Figure~\ref{fig:feas_sol}, the sortie of drone 2 starts in position $2$ and ends in position $3$, because the truck traverses the second arc of its route immediately after launching, and the third arc right before retrieving the drone. Analogously, the sortie of drone 3 starts in position $1$ and ends in position $4$.

We prove that we can solve the TSP-mD by only considering truck routes that contain at most $2\cdot\left| N\right|$ arcs.

\noindent\textbf{Proposition B.1} \emph{There always exists an optimal solution whose truck route contains at most $2\cdot\left| N\right|$ arcs in $N\times N$ and such that no two solution sorties have the same ending position.}

\medskip
\begin{proof}
Let $\left( i,j\right)\in N\times N$ be an arc in an optimal truck route $\pi$. If node~$j$ was not previously visited, then the truck traversal of $\left( i,j\right)$ contributes to visit at least one new node. If instead node~$j$ was already visited, then, by Lemma~\ref{lemma1}, we can assume that there is at least one drone that either lands or takes off at~$j$ immediately after the truck has traversed $\left( i,j\right)$. Thus, the truck traversal of $\left( i,j\right)$ on average contributes to visit half a new node, because every sortie has two endpoints and serves at least one node. Hence, there is always an optimal solution $\mathcal{S}$ whose truck route contains at most $2\cdot\left| N\right|$ arcs.

Notice that, given an optimal solution, we can always construct an optimal solution such that no two solution sorties contain the same ending position. Indeed, if multiple drones are planned to land in the same position, it suffices to expand the truck route with further arcs in $N\times N\smallsetminus A$: we then get another optimal solution, to which the argument above still applies, because by Lemma~\ref{lemma1} there still is at least one sortie starting or ending at every revisit of a node.
\end{proof}

Requiring that no two solution sorties have the same ending position is crucial to simplify our MILP formulation. We denote the position of an arc in a truck route by $t\in \left\lbrace 1, \ldots, T\right\rbrace$, with $T=2\cdot\left| N\right|$. We define a \emph{drone operation} $h$ as a $4$-tuple $\left( d, \pi, t^{1}, t^{2}\right)$ made by a drone $d\in \{1, \ldots, m\}$, a feasible sortie $\pi\in P$, a starting position $t^{1}\in\left\lbrace 1, \ldots, T\right\rbrace$, and an ending position $t^{2}\in\left\lbrace t^{1}, \ldots, T\right\rbrace$. The set of all the drone operations is denoted by $H$.

Our decision variables can be described as follows. For every arc $a\in N\times N$ and for every position $t\in\left\lbrace 1, \ldots, T\right\rbrace$, the binary variable $x_{at}$ is $1$ if and only if arc $a$ is in position $t$ in the truck route. For every position $t\in\left\lbrace 1, \ldots, T\right\rbrace$, the continuous variable $w_{t}\geq 0$ represents the amount of time the truck waits for a drone at the destination node of the arc in position $t$. For every drone operation $h\in H$, the binary variable $z_{h}$ is $1$ if and only if $h$ is performed in the solution. We assume that $N=N^{\mathit{dr}}=N^{\mathit{tr}}$; if not, it is easy to impose the following constraints:
\begin{flalign}
\sum_{t\in\left\lbrace 1, \ldots, T\right\rbrace}\;\sum_{j\in N\smallsetminus\left\lbrace i\right\rbrace}x_{ijt}\leq 0 && \forall i\in N\smallsetminus N^{\mathit{tr}},
\end{flalign}
\begin{flalign}
\sum_{t\in\left\lbrace 1, \ldots, T\right\rbrace}\;\sum_{i\in N\smallsetminus\left\lbrace j\right\rbrace}x_{ijt}\geq 1 &&\forall j\in N\smallsetminus N^{\mathit{dr}}.
\end{flalign}

For sake of conciseness, we slightly abuse the notation of a number of subsets of $H$; this never gives rise to ambiguities, because we always denote drones and starting and ending positions by the symbols $d$, $t_{1}$ and $t_{2}$, respectively. Table \ref{tab:symbols} provides a description of every such subset of $H$.

\begin{table}[ht]
\footnotesize
\centering
\begin{tabular}{ll}
\multicolumn{2}{l}{\textbf{Subsets of $H$}}\\\hline
$H^{k}$ & set of drone operations in $H$ whose sorties are in $P^{k}$ with $k\in N^{\mathit{dr}}$\\
$H_{t^{1}\cdot}$ & set of drone operations in $H$ that start at period $t^{1}$\\
$H_{\cdot t^{2}}$ & set of drone operations in $H$ that end at period $t^{2}$\\
$H_{t^{1} t^{2}}$ & intersection of sets $H_{t^{1}\cdot}$ and $H_{\cdot t^{2}}$\\
$H_{d, t^{1} t^{2}}$ & set of drone operations in $H_{t^{1}t^{2}}$ performed by drone $d$\\
$H_{t^{1}\cdot,i\cdot}$ & set of drone operations in $H_{t^{1}\cdot}$ whose sorties start at node $i\in N^{\mathit{tr}}$\\
$H_{\cdot t^{2},\cdot j}$ & set of drone operations in $H_{\cdot t^{2}}$ whose sorties end at node $j\in N^{\mathit{tr}}$
\end{tabular}
\caption{Description of the subsets of $H$ that are relevant to our formulation of the TSP-mD.}
\label{tab:symbols}
\end{table}

{\small
\begin{lpformulation}[]
\lpobj{min}{\label{obj}\sum_{t\in\left\lbrace 1, \ldots, T\right\rbrace}\left( w_{t}\; +\; \sum_{a\in A}\ell_{a}x_{at}\right)}
\lpeq{\label{onearcpertime}\sum_{a\in N\times N}\; x_{at} \leq 1&&}{t\in\left\lbrace 1, \ldots, T\right\rbrace}
\lpeq{\label{TSP2}\sum_{j\in N}x_{ij1}=\begin{cases*}1 & if $i=0$\\ 0 & otherwise\end{cases*}&&}{i\in N}
\smallskip
\lpeq{\label{TSP3}\sum_{j\in N}x_{ijt}\leq\sum_{j\in N}x_{ji,t-1}&&}{i\in N,\;\;\forall t\in\left\lbrace 2, \ldots, T\right\rbrace}
\lpeq{\label{TSPend}\sum_{i\in N}\;\sum_{j\in N\smallsetminus\left\lbrace 0\right\rbrace} x_{ijt}\leq\begin{cases*}\sum_{a\in N\times N}\; x_{a, t+1}& if $t\neq T$\\ 0 & otherwise\end{cases*}&&}{t\!\in\!\left\lbrace 1, \ldots, T\right\rbrace}
\smallskip
\lpeq{\label{u4}\sum_{t^{1}\in\left\lbrace 1, \ldots, t\right\rbrace}\;\sum_{t^2\in\left\lbrace t, \ldots, T\right\rbrace}\;\sum_{h\in H_{d,t^{1}t^{2}}}z_{h}\leq \sum_{a\in N\times N}x_{at}&&}{d\in D,\;\;\forall t\in\left\lbrace 1, \ldots, T\right\rbrace}
\medskip
\lpeq{\label{u5}\sum_{h\in H_{t^{1}\cdot,i\cdot}}z_{h}\leq m\cdot\sum_{j\in N}x_{ijt^{1}}&&}{i\in N,\;\;\forall t^{1}\in\left\lbrace 1, \ldots, T\right\rbrace}
\lpeq{\label{u6}\sum_{h\in H_{\cdot t^{2},\cdot j}}z_{h}\leq \sum_{i\in N}x_{ijt^{2}}&&}{j\in N,\;\;\forall t^{2}\in\left\lbrace 1, \ldots, T\right\rbrace}
\lpeq{\label{u7}\sum_{t\in\left\lbrace 1, \ldots, T\right\rbrace}\;\sum_{j\in N}x_{kjt}\; +\;\sum_{h\in H^{k}}\; z_{h}\geq 1&&}{k\in N}
\lpeq{\label{u8}\sum_{t\in\left\lbrace t^{1}, \ldots, t^{2}\right\rbrace} \left( w_{t}+\sum_{a\in A}\ell_{a}x_{at}\right)\geq\sum_{h\in H_{t^{1}t^{2}}}\ell'_{h}z_{h}&&}{t^{1}\!\in\!\left\lbrace 1, \ldots, T\right\rbrace,\;\;\forall t^{2}\!\in\!\left\lbrace t^{1}, \ldots, T\right\rbrace}
\lpeq{x_{at}\in\left\lbrace 0,1\right\rbrace&&}{a\in A,\;\;\forall \left\lbrace 1, \ldots, T\right\rbrace}
\lpeq{z_{h}\in\left\lbrace 0,1\right\rbrace&&}{h\in H}
\lpeq{w_{t}\geq 0&&}{t\in\left\lbrace 1, \ldots, T\right\rbrace}
\end{lpformulation}
}

The objective function~\eqref{obj} quantifies the completion time. Constraints~\eqref{onearcpertime} impose that at most one arc in $N\times N$ can be in position $t\in\left\lbrace 1, \ldots, T\right\rbrace$ in the truck route. Constraints~\eqref{TSP2} force that, in position $1$ of the truck route, the depot has one outgoing arc, while all the other nodes $i\in N\smallsetminus\left\lbrace 0\right\rbrace$ do not have any. For all the positions $t\in\left\lbrace 2, \ldots, T\right\rbrace$, by constraints~\eqref{TSP3}, the total outgoing flow of the $x$ variables from a given node $i\in N$ in $t$ is smaller or equal to the total inflow into $i$ in position $t-1$. For any position $t\in\left\lbrace 1, \ldots, T-1\right\rbrace$, if the truck route does not have any arc in position $t+1$, then, by constraints~\eqref{TSPend}, in position $t$ either the truck does not move or it travels directly to the depot. Constraints~\eqref{u4} make sure that, if the truck route actually contains an arc in position $t\in\left\lbrace 1, \ldots, T\right\rbrace$, every drone $d\in D$ is performing at most one sortie. Constraints~\eqref{u5} and~\eqref{u6} force a sortie to start and end where and when the truck is located. We further impose that at most one drone can land in any position $t\in\left\lbrace 1, \ldots, T\right\rbrace$; by Proposition~\ref{app_model}.1, this condition is met by at least one optimal solution, and at the same time it simplifies constraints~\eqref{u8}, whose right-hand side can now be expressed as a summation over the drone operation index $h$. Constraints \eqref{u7} make sure that all the nodes are visited by either the truck or a drone. Finally, constraints~\eqref{u8} synchronize the truck and the drones by giving to the variables $w$ their meaning as per definition, with $\ell'_{h}$ denoting the length of the sortie in operation $h\in H$.

\end{document}